\documentclass[11pt]{amsart}

\usepackage{bm,latexsym,amsfonts,amsmath,amssymb,amsthm,url,graphics,psfrag,amscd,stmaryrd, mathrsfs}
\usepackage[english]{babel}
\usepackage[T1]{fontenc}
\usepackage{float}
\usepackage{graphics}
\usepackage{graphicx}
\usepackage{color}
\usepackage[colorinlistoftodos]{todonotes} 
\usepackage{hyperref}

\newcommand{\conn}{\longleftrightarrow}
\usepackage{mathrsfs}

\newcommand{\src}{\mathscr{C}}

\newcommand{\R}{\mathbb{R}}

\newcommand{\E}{\mathbb{E}}

\newcommand{\pp}{\mathbb{P}}

\newcommand{\kC}{\mathcal{C}}

\newcommand{\kP}{\mathcal{P}}

\newcommand{\kS}{\mathcal{S}}

\newcommand{\kL}{\mathcal{L}}

\newcommand{\bB}{\boldsymbol{B}}
\newcommand{\lin}{\left[\kern-0.15em\left[}
\newcommand{\rin} {\right]\kern-0.15em\right]}
\newcommand{\linf}{[\kern-0.15em [}
\newcommand{\rinf} {]\kern-0.15em ]}
\newcommand{\ilin}{\left]\kern-0.15em\left]}
\newcommand{\irin} {\right[\kern-0.15em\right[}

\newtheorem{theorem}{Theorem}[section]

\newtheorem{lemma}[theorem]{Lemma}
\newtheorem{corollary}[theorem]{Corollary}

\newtheorem{conjecture}[theorem]{Conjecture}

\newtheorem{remark}[theorem]{Remark}
\newtheorem{cond}[theorem]{Condition}

\newcommand{\be}{\begin{equation}}
\newcommand{\ee}{\end{equation}}
\newcommand{\beal}{\begin{align}}
\newcommand{\eal}{\end{align}}
\newcommand{\nn}{\nonumber}

\newcommand{\expec}{\mathbb E}

\newcommand{\sss}{\scriptscriptstyle}

\numberwithin{equation}{section}

\newcommand{\e}{{\rm e}}

\newcommand{\rem}[1]{}

\newcommand{\RvdH}[1]{\todo[inline, color=magenta]{{\rm Remco: #1}}}

\newcommand{\Hao}[1]{\todo[inline, color=green]{#1}}

\def\1{{\mathchoice {1\mskip-4mu\mathrm l}      
{1\mskip-4mu\mathrm l}
{1\mskip-4.5mu\mathrm l} {1\mskip-5mu\mathrm l}}}
\newcommand{\indic}[1]{\1_{\{#1\}}}

\newcommand{\eqn}[1]{\begin{equation}#1\end{equation}}
\newcommand{\eq}[1]{\begin{equation*}#1\end{equation*}}
\newcommand{\eqa}[1]{\begin{align*}#1\end{align*}}
\newcommand{\eqan}[1]{\begin{align}#1\end{align}}

\newcommand{\convp}{\stackrel{\sss {\mathbb P}}{\longrightarrow}}
\newcommand{\convd}{\stackrel{\sss {d}}{\longrightarrow}}

\newcommand{\bsigma}{\boldsymbol{\sigma}}
\newcommand{\bomega}{\boldsymbol{\omega}}

\newcommand{\bxi}{\boldsymbol{\xi}}
\newcommand{\bw}{\boldsymbol{w}}

\DeclareSymbolFont{extraup}{U}{zavm}{m}{n}
\DeclareMathSymbol{\varheart}{\mathalpha}{extraup}{86}
\DeclareMathSymbol{\vardiamond}{\mathalpha}{extraup}{87}
\newcommand{\ensymboldefinition}{$\blacktriangleleft$}
\newcommand{\vertex}{o}
\newcommand{\Tree}{{\sf T}}

\usepackage{fancyhdr}
\newcommand{\invisible}[1]{}
\newcommand{\invisibleRvdH}[1]{}

\begin{document}

\title{Random cluster models on random graphs}
\author{Van Hao Can }
\address{ 
Institute of Mathematics, Vietnam Academy of Science and Technology, 18 Hoang Quoc Viet, 10307 Hanoi, Vietnam, cvhao@math.ac.vn}
\author{Remco van der Hofstad}
\address{Remco van der Hofstad, Eindhoven University of Technology, P.O. Box 513, 5600 MB Eindhoven, The Netherlands, r.w.v.d.hofstad@tue.nl}


\begin{abstract}
On locally tree-like random graphs, we relate the random cluster model with external magnetic fields and $q\geq 2$ to Ising models with vertex-dependent external fields. The fact that one can formulate general random cluster models in terms of two-spin ferromagnetic Ising models is quite interesting in its own right. However, in the general setting, the external fields are both positive and negative, which is mathematically unexplored territory. Interestingly, due to the reformulation as a two-spin model, we can show that the Bethe partition function, which is believed to have the same pressure per particle, is always a {\em lower bound} on the graph pressure per particle. We further investigate special cases in which the external fields do always have the same sign. 

The first example is the Potts model with general external fields on random $d$-regular graphs. In this case, we show that the pressure per particle in the quenched setting agrees with that of the annealed setting, and verify \cite[Assumption 1.4]{BasDemSly23}. We show that there is a line of values for the external fields where the model displays a first-order phase transition. This completes the identification of the phase diagram of the Potts model on the random $d$-regular graph. As a second example, we consider the high external field and low temperature phases of the system on locally tree-like graphs with general degree distribution.

\end{abstract}

\maketitle


\section{Random cluster models}
\label{sec-|R|C}
In this section, we motivate the problem, introduce the models, and state our main results.
\subsection{Introduction and motivation}
\label{sec-motivation}
Random cluster models on random graphs have received significant attention in the literature as a key example of percolation models with {\em dependent} edges. See \cite{Grim06} for an exposition of the random cluster model, also sometimes called FK-percolation after the pioneers Fortuin and Kasteleyn \cite{ForKas72}, who invented the model.  We also refer to \cite{Fort72a,Fort72b} for further properties of the random cluster model. In particular, the random cluster model can be used to study various important models in statistical mechanics in one unified way. The models include percolation for $q=1$, the Ising model for $q=2$, and the general $q$-state Potts model for $q\geq 3$. These models are paradigmatic models in statistical mechanics in that they all display a {\em phase transition} on most infinite graphs, and this phase transition can have rather different shapes. 
\smallskip

In this paper, we generalise the recent results of Bencs, Borb\'enyi and Csikv\'ary \cite{BenBorCsi23} that translates the partition function of the random cluster model on general graphs having not too many short cycles to include a magnetic field. There are many papers that study Ising, Potts and random cluster models on random graphs. For the Ising model, let us refer to \cite{Can17,Can19,CanGiaGibHof22,DemMon10b,DemMon10a,DemMonSun13,DomGiaHof10,DomGiaHof12} to mention just a few. An overview of results is given in \cite{Hofs25}. For the Ising model, the picture is quite clear. The thermodynamic limit of the logarithm of the partition function is proved to exist for general parameters, as well as general random graphs models whose local limit is a (random) tree. From this, the nature of the phase transition can be deduced. 
\smallskip

While there are some results on Potts and random cluster models \cite{BasDemSly23,BenBorCsi23,DemMonSlySun14,GiaGibHofJanMai25,HelJenPer23}, the picture there is less complete. In particular, all results on random graphs apply to the random $d$-regular graph, or to annealed settings. It is shown that the pressure per particle exists in the absence of a magnetic field, but not yet in the presence of an external field. This makes it more difficult to study the nature of the phase transition in terms of the existence of a dominant state. Potts and random cluster models have also been studied from the combinatorial perspective, due to their links with Tutte polynomials,  and the algorithmic perspective, in terms of the complexity of computing the partition function. We refer to \cite{BenBorCsi23,HelJenPer23} for extensive literature overviews on these topics.

\subsection{Model}
\label{sec-model}
In this section, we state the models that we will consider. We start by describing the random cluster model, followed by a discussion of locally tree-like graphs.
\smallskip

\paragraph{\bf The random cluster model.} We follow \cite{Grim06}. Let $G=(V(G),E(G))$ be a finite graph with vertex set $V(G)$ and edge set $E(G)$. We assume that $|V(G)|=n$. The random cluster measure is denoted by $\phi_{p,q}$, and gives a measure on configurations $\bomega=(\omega(e))_{e\in E}$, where the edge variables $\omega(e)$ satisfy $\omega(e)\in \{0,1\}$. We think of $e$ for which $\omega(e)=1$ as being {\em present,} and $e$ for which $\omega(e)=0$ as being {\em vacant}. The random cluster measure $\phi_{p,q}$ is given by
	\eqn{
	\label{RC-measure}
	\phi_{p,q}(\bomega)=\frac{1}{Z_{n,\rm RC}(p,q)} q^{k(\bomega)} \prod_{e\in E(G)} p^{\omega(e)}(1-p)^{1-\omega(e)},
	}
where $q>0$ and $p\in [0,1]$, and $k(\bomega)$ denotes the number of connected components of the subgraph $G_p$ with vertex set $V(G)$ and edge set $E_p(G)=\{e\in E(G)\colon \omega(e)=1\}$. The partition function, or normalising constant, $Z_{n,\rm RC}(p,q)$ is given by
	\eqn{
	\label{RC-partition}
	Z_{n,\rm RC}(p,q)=\sum_{\bomega} q^{k(\bomega)} \prod_{e\in E(G)} p^{\omega(e)}(1-p)^{1-\omega(e)},
	}
where the sum is over all bond variables $\omega(e)\in \{0,1\}$.
\medskip

\paragraph{\bf The relation to Potts models.} In the above, $q>0$ is general, as it will often be in this paper. We now specialise to $q\in \{1,2,3\ldots\}$, and relate the above model to the Potts model. For this, we need to specify vertex variables $\bsigma=(\sigma_x)_{x\in V(G)}\in [q]^{V(G)}$, where $[q]=\{1, \ldots, q\}$. We first define a {\em joint law} $\mu_{p,q}(\bsigma,\bomega)$ given by
	\eqn{
	\label{RC-B=0}
	\mu_{p,q}(\bsigma,\bomega)=\frac{1}{Z_{n}(p,q)} \prod_{\{x,y\}\in E(G)} 
	\Big[(1-p)\indic{\omega(\{x,y\})=0}+p\indic{\omega(\{x,y\})=1}\indic{\sigma_x=\sigma_y}\Big],
	}
where the partition function, or normalising constant, $Z_{n}(p,q)$ is now given by
	\eqn{
	\label{RC-B=0-Z}
	Z_{n}(p,q)=\sum_{\bsigma, \bomega}  \prod_{\{x,y\}\in E(G)} 
	\Big[(1-p)\indic{\omega(\{x,y\})=0}+p\indic{\omega(\{x,y\})=1}\indic{\sigma_x=\sigma_y}\Big],
	}
and the sum is over all bond variables $\omega(e)\in \{0,1\}$ and vertex variables $\sigma_v\in [q]$ for all $v\in V(G)$.
\smallskip

The key coupling result is that (cf.\ \cite[Theorem (1.10)]{Grim06}) for $p=1-\e^{-\beta}$ and $q\in {\mathbb N}$, the spin marginal $\bsigma \mapsto \mu_{p,q}(\bsigma)$ obtained by summing out over the bond variables is the Potts model with parameters $\beta$ and $q$, while the bond marginal $\bomega \mapsto \mu_{p,q}(\bomega)$ obtained by summing out over the vertex spins is the random cluster measure. 
\medskip

Equations \eqref{RC-B=0}--\eqref{RC-B=0-Z} define the random cluster and Potts model without an external field, and we now extend it to an external field. For this, we define the Potts model partition function, with inverse temperature $\beta\geq 0$ and external field $B$, as
	\eqn{
	\label{NC-Potts}
	Z_G^{\sss \rm Potts}(q,\beta,B)=\sum_{\bsigma\in [q]^{V(G)}} \e^{-H(\bsigma)},
	}
where 
	\eqn{
	\label{Hamiltonian-Potts}
	H(\bsigma)=H_{\beta,B}(\bsigma)=-\beta\sum_{\{u,v\}\in E(G)} \indic{\sigma_u=\sigma_v}-B\sum_{v\in V(G)} \indic{\sigma_v=1}.
	}
Note that $Z_G^{\sss \rm Potts}(q,\beta,0)=\e^{-\beta|E(G)|}Z_{n}(p,q)$ with $p=1-\e^{-\beta}$, but that will play no role in what follows. In this paper, we investigate this partition function, as well as its random cluster model extension, which we define next. 
\smallskip
	
\paragraph{\bf The random cluster model with external field.} We next define the main object of study, which is the random cluster model with a weight that we think of as an {\em external field}. For Potts and Ising models, this weighted model is exactly the Potts/Ising model with an external field, while the definition extends to non-integer $q$. Let $q>1$ be a positive real number, and $w=\e^{\beta}-1$. Then, we define the random cluster measure with external field $B$, inspired by \cite{BisBorChaKot00},  by
	\eqn{ 
	\label{rcb}
	\mu_G(\bomega) = \frac{W(\bomega)}{Z_G(q,w,B)}, \qquad \text{and} \qquad
	W(\bomega)=\prod_{e\in E(G)} w^{\omega_e}
	\prod_{C\in\mathscr{C}(\bomega)}
	(1+(q-1)\e^{-B|C|}),
	}
where the normalisation constant $Z_G(q,w,B)$ is given by
	\eqn{
	\label{partition-function-B}
	Z_G(q,w,B)=\sum_{\bomega}\prod_{e\in E(G)} w^{\omega_e}
	\prod_{C\in\mathscr{C}(\bomega)}
	(1+(q-1)\e^{-B|C|}),
	}
where $\mathscr{C}(\bomega)$ is the collection of connected components in $\bomega$. It is this measure that we will focus on in this paper. The following lemma relates the partition function for the Potts model with external field $B$ to $Z_G(q,w,B):$

\begin{lemma}[Relation Potts and random cluster partition function]
\label{pott=rc}
Let $q\geq 2$ be an integer, and let $G=(V(G), E(G))$ be a finite graph. Then, with $\beta =\log(1+w)$,
		\eq{
			Z_G^{\sss \rm Potts}(\beta,B) = \e^{B|V(G)|} Z_G(q,w,B).
	}
	\end{lemma}

We next describe the random graph models that we will be working on.
\medskip

\paragraph{\bf Random graphs and local convergence.} In the above, the finite graph $G=(V(G),E(G))$ is general. Here, we introduce the kind of random graphs that we shall work on.
We will consider graph sequences $(G_n)_{n\geq}$ of growing size, where $G_n=(V(G_n),E(G_n))$, with $V(G_n)$ and $E(G_n)$ denoting the vertex and edge sets of $G_n$, respectively. 
We assume that $|V(G_n)|=n$, and assume that $n$ tends to infinity. 
\begin{cond}[Uniform sparsity]
\label{cond-uniform-sparse}
{\rm We assume that $(G_n)_{n\geq}$ are {\em uniformly sparse}, in the sense that $D_n$, the degree of a random vertex in $V(G_n)$, satisfies that
	\eqn{
	\label{degree-convergence}
	D_n\convd D,
	\qquad
	\expec[D_n]\rightarrow \expec[D],
	}
for some limiting degree distribution $D$.}\hfill \ensymboldefinition
\end{cond}
\smallskip

We will strongly rely on {\em local convergence}, as introduced by Benjamini and Schramm in \cite{BenSch01}, and, independently by Aldous and Steele in \cite{AldSte04}. We refer to \cite[Chapter 2]{Hofs24} for an overview of the theory, and  \cite[Chapters 3-5]{Hofs24} for several examples. By \cite{BenLyoSch15}, \eqref{degree-convergence} implies that the graph sequence $(G_n)_{n\geq}$ is precompact in the local topology, meaning that every subsequence has a further subsequence that converges in the local topology. See also \cite{AldLyo07} for further discussion. Our main assumption is that the graph sequence $(G_n)_{n\geq}$ is {\em locally tree-like}:

\begin{cond}[Locally tree-like random graph sequences]
\label{cond-locally-tree-like}
{\rm We assume that $(G_n)_{n\geq}$ are {\em locally tree-like}, in the sense that $G_n$ converges locally in probability to a rooted graph $(G,\vertex)$ that is almost surely a tree.}\hfill \ensymboldefinition
\end{cond}
\smallskip

Conditions \ref{cond-uniform-sparse} and \ref{cond-locally-tree-like} have a long history in considering Ising and Potts models on random graphs. See e.g., \cite{DemMon10a,DemMon10b,DemMonSlySun14,DomGiaHof10} for some of the relevant references. Below, and throughout the paper, we call a graph sequence $(G_n)_{n\geq}$ {\em locally tree-like} when Conditions \ref{cond-uniform-sparse} and \ref{cond-locally-tree-like} both hold.
\medskip

\paragraph{\bf Organisation of this section} This section is organised as follows. We start by describing our main results, which are divided into three parts. In Section \ref{sec-main-results}, we rewrite the random cluster partition function with general external field $B>0$ in terms of an extended Ising model with specific vertex-dependent external fields. In Section \ref{sec-Bethe}, we use this representation to give a lower bound on the partition function by the Bethe partition function. In Section \ref{sec-|R|C-d-regular}, we use these results to analyse the phase transition of the Potts model on random $d$-regular graphs. In Section \ref{sec-eIsing-pressure}, we describe the results for the extended Ising model, and their consequences on the random cluster model for high external magnetic fields and low temperature. We close in Section \ref{sec-discussion-open-problems} by discussing  our results and stating open problems.

\subsection{Main results for general graphs}
\label{sec-main-results}
In this section, we explain how the partition function of the random cluster model with non-negative external field can be rewritten in terms of an Ising model with vertex-dependent external fields.
\medskip

\paragraph{\bf Rank-2 approximation to random cluster model with external fields.}
We first give bounds on the partition function of the random cluster model with external field, extending the work of Bencs, Borb\'enyi and Csikv\'ary \cite{BenBorCsi23} to non-zero external fields:

\begin{theorem}[Rank-2 approximation of random cluster partition function with external field]
		\label{thm-B>0}
		Let $G=(V(G),E(G))$ be a finite graph. Then if $q\geq 2$ and $B \geq 0$,
		\eqn{
			Z^{\sss (2)}_G(q,w,B)\leq Z_G(q,w,B)\leq q^{\kL(G)}Z^{\sss (2)}_G(q,w,B),
		}
		where 
		\eqn{ \label{mdcy}
			\kL(G)= \max_{A \subseteq E} |\{ \text{\rm connected components of $G_A=(V(G),A)$ containing a cycle}\}|,
		}
		and
		\eqan{
			\label{Z-2-form}
			Z^{\sss (2)}_G(q,w,B)
			&= \sum_{S\subseteq V(G)}(1+w)^{|E(S)|}\left(1+\frac{w}{q-1}\right)^{|E(V(G)\setminus S)|} ((q-1)\e^{-B})^{|V(G)|-|S|}.
		}
\end{theorem}
The proof of Theorem \ref{thm-B>0} is given in Section \ref{sec-quenched-random-cluster}, and closely follows Bencs, Borb\'enyi and Csikv\'ary \cite{BenBorCsi23}. 
\medskip

We next use Theorem \ref{thm-B>0} to investigate the pressure per particle of the random cluster model. We can trivially bound
        \eqn{
        \kL(G) \leq \frac{|V(G)|}{k} + \sum_{i=2}^{k-1} \kL_i(G),
        }
where 
        \eqn{
        \kL_i(G) = \max\{l\colon  G~\text{\rm contains $l$ vertex-disjoint cycles of length $i$}\}.
        }
In particular, for graphs $G_n=(V(G_n),E(G_n))$ on $|V(G_n)|=n$ vertices, having a locally tree-like limit, we can bound
	\eqn{
	\frac{1}{n}\kL(G_n) \leq \frac{1}{k} + \frac{1}{n}\sum_{i=2}^{k-1} \kL_i(G_n)\convp \frac{1}{k},
	}
since $\kL_i(G_n)/n\convp 0$ by the fact that the local limit is tree-like (recall Condition \ref{cond-locally-tree-like}). Since $k$ is arbitrary, it follows that 
	\eqn{
	\frac{1}{n}\log Z_{G_n}(q,w,B) =\frac{1}{n}\log Z_{G_n}^{\sss (2)}(q,w,B) +o(1).
	}
Thus, the pressure per particle for the random cluster model is the same as that for the rank-2 approximation, as will be a guiding principle for this paper:

\begin{corollary}[Rank-2 approximation of random cluster partition function with external field]
\label{cor-B>0}
		Let $G_n=(V(G_n),E(G_n))$ be a finite graph of size $|V(G_n)|=n$. Then if $q\geq 2$ and $B \geq 0$,
		\eqn{
		\lim_{n\rightarrow \infty} \frac{1}{n} \log Z_{G_n}(q,w,B)
		=\lim_{n\rightarrow \infty} \frac{1}{n} \log Z^{\sss(2)}_{G_n}(q,w,B).
		}
\end{corollary}
We continue by investigating the rank-2 approximation $Z^{\sss(2)}_{G_n}(q,w,B)$ in more detail.
\medskip

\paragraph{\bf Rewrite of rank-2 approximation in terms of two-spin models.} We next write the expressions in \eqref{Z-2-form} in terms of two-spin models, which will turn out to be general Ising models with vertex-dependent external fields. 
Define
	\eqn{
	\label{ZG-psi-def}
	Z_G(\psi,\overline{\psi})=\sum_{\bsigma\in \{+,-\}^{V(G)}}\prod_{v\in V(G)} \overline{\psi}(\sigma_v)\prod_{\{u,v\}\in E(G)} \psi(\sigma_u,\sigma_v),
	}
where we use a slight abuse of notation and write $+$ instead of $+1$ and $-$ instead of $ -1$. Sly and Sun \cite{SlySun14} show that any two-spin model on regular graphs can be mapped to an Ising model, which can be ferro- or antiferromagnetic. Our second main result rewrites the partition function of the rank-2 approximation of the random cluster model with inverse temperature $\beta$ and external field $B$ in terms of the partition function of a ferromagnetic two-spin model:

\begin{theorem}[Rewrite in terms of two-spin models]
\label{thm-gen-Ising-Z(2)}
Let $G=(V(G),E(G))$ be a finite graph. For every $q,w,B$, 
		\eqn{
			Z^{\sss (2)}_G(q,w,B) = Z_G(\psi,\overline{\psi}),
		}
		where
		\eqn{
			\label{psi-def}
			\psi(+,+)=1+w,
			\quad
			\psi(+,-)=\psi(-,+)=1,
			\quad
			\psi(-,-)=1+w/(q-1),
		}
		and 
		\eqn{
			\label{overline-psi-def}
			\overline{\psi}(+)=1, \qquad \overline{\psi}(-)=(q-1)\e^{-B}.
		}
\end{theorem}
With Theorem \ref{thm-gen-Ising-Z(2)} in hand, we now set to reformulating this result in terms of ferromagnetic Ising models. 
\medskip

\paragraph{\bf Relation to general Ising models.} We next turn to Sly and Sun \cite{SlySun14}, who investigate general two-spin models, and relate them to Ising models. 
Note that 
	\eqn{
	\label{ZG-psi-def-rew}
	Z_G(\psi,\overline{\psi})=(\overline{\psi}(+)\overline{\psi}(-))^{n/2} \sum_{\bsigma\in \{-1,1\}^{V(G)}}\prod_{v\in V(G)} \e^{h\sigma_v} \prod_{\{u,v\}\in E(G)} \psi(\sigma_u,\sigma_v),
	}
where
	\eqn{
	\e^{2h} = \frac{\overline{\psi}(+)}{\overline{\psi}(-)},
	}
so that
	\eqn{
	\label{h-choice}
	h=\frac{1}{2} \log (\overline{\psi}(+)/\overline{\psi}(-))=\frac{1}{2} \log (\e^B/(q-1)).
	}
Then, since $\psi>0$,  we can rewrite (see \cite[page 2393]{SlySun14})
	\eqn{
	\label{psi-two-def}
	\psi(\sigma,\sigma')=\e^{B_0} \e^{\beta^* \sigma\sigma'}\e^{k(\sigma+\sigma')},
	}
where the parameters $B_0$, $k$ and $\beta^*$ are determined by
	\eqan{
	\frac{\psi(+,+)}{\psi(-,-)}&=\e^{4k},
	\label{B-restriction}\\
	\frac{\psi(+,+)\psi(-,-)}{\psi(+,-)^2}&=\e^{4\beta^*},
	\label{beta-restriction}\\
	\psi(+,+)\psi(+,-)^2\psi(-,-)&=\e^{4B_0}.
	\label{B0-restriction}
	}
Thus,
	\eqn{
	\label{k-choice}
	k=\frac{1}{4}\log \left(\frac{1+w}{1+w/(q-1)}\right),
	}
and, by \eqref{psi-def} and \eqref{beta-restriction},
	\eqn{
	\label{beta-choice-a}
	(1+w/(q-1))(1+w)=\e^{4\beta^*},
	}
so that
	\eqn{
	\label{beta-choice}
	\beta^*= \frac{1}{4} \log ((1+w)(1+w/(q-1))),
	}
while, by \eqref{psi-def} and \eqref{B0-restriction},
	\eqn{
	(1+w/(q-1))(1+w)=\e^{4B_0},
	}
so that
	\eqn{ 
	\label{B0-choice}
	B_0= \frac{1}{4} \log ((1+w)(1+w/(q-1)))=\beta^*.
	}
\medskip

 We conclude that
	\eqn{
	\label{ZG-eising}
	Z_G(\psi,\overline{\psi})=(\e^{-B}(q-1))^{n/2} \e^{\beta^*|E(G)|}Z_{G}^{\sss \rm eIsing}(\beta^*,k,h),
	}
where we define the partition function of the \textit{extended Ising model} as 
	\eqn{ \label{def-eising}
	Z_{G}^{\sss \rm eIsing}(\beta^*,k,h):=\sum_{\bsigma\in \{-1,1\}^{V(G)}}\exp \left( \beta^* \sum_{\{u,v\}\in E(G)} \sigma_u \sigma_v + \sum_{v \in V(G)} (k d_v+h) \sigma_v \right),
	}
and $d_v=d_v(G)$ is the degree of vertex $v\in V(G)$. This Ising model is unusual, due to the vertex-dependent external fields, which are governed by the {\em degrees} of the vertices in the graph. We defer a discussion of the parameters to below Corollary \ref{thm-gen-Ising}.
\medskip

The above computations lead to the following two corollaries:

\begin{corollary}[Rewrite in terms of extended Ising models]
\label{thm-gen-Ising}
Let $G=(V(G),E(G))$ be a finite graphs with $|V(G)|=n$ vertices. For every $q,w,B$, 
	\eqan{
	\label{Z2-sum-eIsing}
	Z^{\sss (2)}_G(q,w,B)= (\e^{-B}(q-1))^{n/2} \e^{\beta^*|E(G)|}Z_{G}^{\sss \rm eIsing}(\beta^*,k,h),
	}
where $k$ is given in \eqref{k-choice}, $\beta^*$ in \eqref{beta-choice}, and $h$ in \eqref{h-choice}.
\end{corollary}

\begin{corollary}[Thermodynamic limit of random cluster model]
\label{cor-termodynamic-limit-|R|C}
Let $G_n=(V(G_n),E(G_n))$ be a random  graph with $|V(G_n)|=n$ vertices that is locally tree-like. Then, if $q\geq 2$ and  $B \geq 0$,
	\eqn{
		\label{thermodynamic-limit-|R|C-q>2}
		\lim_{n\rightarrow \infty} \frac{1}{n} \log Z_{G_n}(q,w,B)=\frac{\beta^*}{2}\expec[D_o]+\frac{1}{2} \log(\e^{-B}(q-1))
		+\lim_{n\rightarrow \infty} \frac{1}{n} \log Z_{G_n}^{\sss \rm eIsing}(\beta^*,k,h),
	}
where $k$ is given in \eqref{k-choice}, $\beta^*$ in \eqref{beta-choice}, and $h$  in \eqref{h-choice}, and provided the limit on the rhs exists.
\end{corollary}

Let us make some observations at this point:

\begin{remark}[Discussion of the parameters]
\label{rem-parameters-eIsing}
{\rm We next discuss the nature of the extended Ising model:
\begin{description}
\item[Sign of $\beta^*$] Since $w=\e^{\beta}-1$, and $\beta>0$, we have that $\beta^*>0$ as well. Thus, our extended Ising model is {\em ferromagnetic;}
\medskip

\item[Prefactors and $B_0$] The factor  multiplying $Z_{G}^{\sss \rm eIsing}(\beta^*,k,h)$ is a constant, and will thus not play a significant role in the properties of our model;
\medskip

\item[Sign of $k$] We have that $k$ in \eqref{k-choice} satisfies that $k>0$ for $q>2$;

\medskip

\item[Sign external fields] The difficulty is that $k$ in \eqref{k-choice}, and $h$ in \eqref{h-choice} might have opposite signs. Thus, the vertex-dependent external field $B_v=k d_v+h$ for $v\in V(G)$ potentially alternates in sign, depending on the degrees. The sign {\em is} fixed when $d_v=d$  for all $v\in V(G)$, i.e., for $d$-regular graphs. The sign is also fixed for $q>2$ when $h\geq 0$, i.e., $B\geq \log (q-1)$, in which case we have a high external field. In all other cases, the sign of the vertex-dependent external field is negative for vertices of small degree, and positive for vertices of large degree. For such mixed ferromagnetic/anti-ferromagnetic models, few results are known;
\medskip

\item[Extended Ising model] The Ising model with vertex-dependent external fields of the form $k d_v+h$ for $v\in V(G)$ has not been investigated in the literature. However, when $k d_v+h\geq B_{\min}>0$, i.e., all vertices have an external field that is bounded below by a positive constant, the methods in \cite{DemMon10b,DemMon10a,DemMonSlySun14,DomGiaHof10} can be used to compute the thermodynamic limit of the partition function. We refer to \cite{Hofs25} for an extensive overview of Ising models on random graphs. \hfill \ensymboldefinition
\end{description}}
\end{remark}

\subsection{Bethe functional on random graphs}
\label{sec-Bethe}
We next use the above representation, combined with a result of Ruozzi \cite{Ruoz12}, to show that the partition function of the random cluster model is bounded below by the Bethe partition function. See also \cite{YedFreWei05} for more details on the Bethe partition function, and \cite{WilSudWai07}, where the authors first conjectured the bound that we are about to discuss. Further, see \cite{MezMon09} for background on belief propagation, and the introduction to \cite{CojGalGolRavSteVig23} for a concise explanation.
\medskip

Before being able to state this bound, let us introduce some notation. Recall \eqref{ZG-psi-def}, and, in terms of this, define $Z_{\rm B}(G,\mu)$ by
	\eqan{
	\label{Bethe-partition-function-mu}
	\log Z_B(G, \mu)&=\sum_{v\in V(G)} \sum_{\sigma\in \{+,-\}} \mu_v(\sigma) \log \overline{\psi}(\sigma)
	+\sum_{\{u,v\}\in E(G)} \sum_{\sigma,\sigma'}\mu_{\{u,v\}}(\sigma,\sigma')\log\psi(\sigma,\sigma')\nn\\
	&\quad-\sum_{v\in V(G)} \sum_{\sigma\in \{+,-\}} \mu_v(\sigma) \log \mu_v(\sigma)
	-\sum_{\{u,v\}\in E(G)} \sum_{\sigma,\sigma'}\mu_{\{u,v\}}(\sigma,\sigma')\log\Big(\frac{\mu_{\{u,v\}}(\sigma,\sigma')}{\mu_u(\sigma)\mu_v(\sigma')}\Big),
	}
where 
	\eqn{
	\mu=\Big((\mu_v(\sigma))_{v\in V(G), \sigma\in \{+,-\}}, (\mu_{e}(\sigma,\sigma'))_{e\in E(G), \sigma, \sigma'\in \{+,-\}}\Big)
	}
is a probability distribution with consistent marginals, i.e., $\mu$ satisfies that $\mu\geq 0$ and $\sum_{\sigma\in \{+,-\}} \mu_v(\sigma)=1$, while, for every $\{u,v\}\in E(G)$ and $\sigma'\in \{+,-\}$,
	\eqn{
	\sum_{\sigma\in \{+,-\}} \mu_{\{u,v\}}(\sigma,\sigma')=\mu_v(\sigma').
	}
In terms of this notation, let
	\eqn{
	\label{Bethe-partition-function}
	Z_B(G)=\max_{\mu} Z_B(G, \mu).
	}

\invisibleRvdH{This can be seen by using Lagrange multipliers on the constraints, as follows. We will differentiate wrt $\mu_v(\sigma)$ and $\mu_{\{u,v\}}(\sigma,\sigma')$, and, to identify the Lagrange multipliers, we write the constraints as
	\eqn{
	\label{Lagrange-constraints}
	\sum_{\sigma} \mu_v(\sigma)-1=0,
	\qquad
	\sum_{\sigma_u \colon \{u,v\}\in E(G)} \mu_{\{u,v\}}(\sigma_u,\sigma')-\mu_v(\sigma')=0.
	}
Then, differentiation wrt $\mu_v(\sigma)$ gives
	\eqn{
	\log \overline{\psi}(\sigma)+\sum_{u\colon \{u,v\}\in E(G)} \sum_{\sigma,\sigma'}\frac{\mu_{\{u,v\}}(\sigma,\sigma')}{\mu_v(\sigma')}=\lambda_v,
	}
while differentiation wrt $\mu_{\{u,v\}}(\sigma,\sigma')$ gives
	\eqn{
	\log \overline{\psi}(\sigma)+\sum_{u\colon \{u,v\}\in E(G)} \sum_{\sigma,\sigma'}\frac{\mu_{\{u,v\}}(\sigma,\sigma')}{\mu_v(\sigma')}=\lambda_v,
	}}

The corollary below relates $Z_G(\psi,\overline{\psi})$ to the Bethe partition function $Z_B(G)$:

\begin{corollary}[Lower bound in terms of Bethe partition function]
\label{cor-Bethe}
For every graph $G$ and $Z_G(\psi,\overline{\psi})$ as in \eqref{ZG-psi-def}, when $\frac{\psi(+,+)\psi(-,-)}{\psi(+,-)^2}\geq 1$,
	\eqn{
	Z_G(\psi,\overline{\psi})\geq Z_B(G).
	}
As a result, for every $q\geq 2$ and $B\geq 0$, also 
	\eqn{
	Z_{G}(q,w,B) \geq Z_B(G).
	}
\end{corollary}
\smallskip

The above equations in \eqref{Bethe-partition-function-mu} and \eqref{Bethe-partition-function} are closely related to {\em belief propagation}, i.e., the solutions of the belief propagation algorithm are solutions to them as well.  For example, we can differentiate $Z_B(G, \mu)$ in \eqref{Bethe-partition-function} wrt the variables $\mu_v(\sigma)$ and $\mu_{\{u,v\}}(\sigma,\sigma')$, and use Lagrange multipliers on the constraints. This will give recursion relations for the variables $\mu_v(\sigma)$ and $\mu_{\{u,v\}}(\sigma,\sigma')$ that are called belief propagation. Such belief propagation equations are {\em exact} on a tree, and one can hope that they are {\em almost} exact on graphs that are close to trees. Corollary \ref{cor-Bethe} shows that the belief propagation solutions always yield {\em lower bounds} on the partition function. See Section \ref{sec-termodynamic-limit-extended-Ising} for belief propagation on trees, where the recursions are {\em exact}.
\smallskip

Let us give some background on the proof of Corollary \ref{cor-Bethe}. We note that
	\eqn{
	Z_{G}^{\sss(2)} (q,w,B)=
	Z_G(\psi,\overline{\psi})=\sum_{\bsigma \in \{+,-\}^{V(G)}} f(\bsigma),
	}
where
	\eqn{
	f(\bsigma)=\prod_{v\in V(G)}\overline{\psi}(\sigma_v)\prod_{\{u,v\}\in E(G)} \psi(\sigma_u,\sigma_v).
	}
For $\bsigma, \bsigma'$, we let $\bsigma\wedge \bsigma'$ and $\bsigma\vee\bsigma'$ be the coordinate-wise minimum and maximum of $\bsigma$ and $\bsigma'$. Then,
	\eqn{
	f(\bsigma\wedge \bsigma')f(\bsigma\vee\bsigma')\geq f(\bsigma)f(\bsigma'),
	}
i.e., the function $f$ is log-supermodular (cf.\ \cite[Definition 2.1]{Ruoz12}).  Ruozzi shows that the lower bound in  Corollary \ref{cor-Bethe} holds for {\em all} log-supermodular functions that are of product structure as in \eqref{ZG-psi-def}. In fact, $\overline{\psi}(\sigma_v)$ may even depend on the vertex $v$ as $\overline{\psi}_v(\sigma_v)$, and $\psi(\sigma_u,\sigma'_v)$ may depend on the edge $\{u,v\}$ as $\psi_{\{u,v\}}(\sigma_u,\sigma'_v)$.
\smallskip

The way how Ruozzi proves this statement in \cite{Ruoz12} is by using $k$-covers, which compare graphs to graphs that consist of several copies of the vertex set, and are such that each vertex in each of the copies has exactly the correct number of edges to copies of its neighbours. The main point is that such covers are more tree-like than the original graph, which may intuitively explain the relation to the Bethe partition function. We refer to \cite{Ruoz12} and the references in it for more details.

\subsection{Results for random cluster models on locally tree-like regular graphs}
\label{sec-|R|C-d-regular}
We first analyse the consequences for the $d$-regular setting, for which the external field has a fixed sign and is thus vertex independent. 
\medskip

\paragraph{\bf Pressure per particle for general external fields.} Let $G_n$ be locally tree-like $d$-regular graph having $|V(G_n)|=n$ vertices. Let us denote $\varphi^{\rm \sss Ising}(\beta,B)$ as the pressure per particle of the Ising model with inverse temperature $\beta\geq 0$ and external field $B\in \mathbb{R}$, i.e., 
	\eqn{
		\frac{1}{n} \log Z_{G_n}^{\sss \rm Ising}(\beta,B)\convp \varphi^{\rm \sss Ising}(\beta,B).
	}
	This limit exists by \cite{DemMon10a}, see also \cite{DemMon10b,DemMonSlySun14,DomGiaHof10,DomGiaHof12} for related results. Our first result shows that the pressure per particle exists for all $\beta,B$:

\begin{theorem}[Thermodynamic limit of random cluster model on regular graphs]
\label{thm-termodynamic-limit-|R|C}
Let $G_n=(V(G_n),E(G_n))$ be a locally tree-like $d$-regular graph having $|V(G_n)|=n$ vertices.  Then, for all $q \geq 2$ and $w, B\geq 0$, the limit $\varphi(w,B)=	\lim_{n\rightarrow \infty} \tfrac{1}{n} \log Z_G(q,w,B)$ is given by
		\eqn{
			\label{thermodynamic-limit-|R|C-q>2-|R|RG}
			\varphi(w,B)= \frac{\beta^*d}{2}+\frac{1}{2} \log(\e^{-B}(q-1))
			+\varphi^{\rm \sss Ising}(\beta^*,B^*)
		}
		where $w=\e^\beta-1$,
		\eqn{
			\beta^*= \frac{\beta}{4} +\frac{1}{4}\log\Big(1+\frac{\e^\beta-1}{q-1}\Big),
		}
		and
		\eqan{
			B^* &=\frac{d}{4}\log \left(\frac{1+w}{1+w/(q-1)}\right)+\frac{1}{2} \log (\e^B/(q-1)).
		}
	\end{theorem}
Due to the link to the Ising model (where now the external fields are constant) in Corollary \ref{cor-termodynamic-limit-|R|C}, this result follows directly from the literature, in particular \cite{DemMon10a}. We refer to \cite{DemMon10b,DemMonSlySun14,DomGiaHof10,DomGiaHof12} for related results, and \cite{Hofs25} for an overview of the Ising model on locally tree-like random graphs. The whole point is that these references prove that the pressure per particle for Ising models exists for {\em all} locally tree-like random graphs, and fixed external fields.
\medskip

\paragraph{\bf The nature of the phase transition.}   We next analyse the phase transition in terms of $\beta$. For $q>2$, the phase transition when  $B=0$ follows from the analysis by Bencs, Borb\'enyi and Csikv\'ary in \cite{BenBorCsi23}. We will show that the first-order phase transition {\em persists} for certain positive external fields $B > 0$. 
\medskip

For $q \geq 2$ and $d\geq 3$, define 
	\eqn{ 
	\label{lqd}
	\ell_{q,d}(x) = \frac{x^{2/d}-1}{1-\tfrac{1}{q-1}x^{2/d}},
	}
and
	\eqn{
	\label{critical-wc}
	 w_c(B):=\ell_{q,d}((q-1)\e^{-B}) 
	}
Then the  {\em critical curve} is defined by 
	\eqn{
	\label{crit-curve}
	\mathscr{C}=\{ (w,B): w=w_c(B), \,  0 \leq  B < B_+\},
	}
where $B_+>0$ is   the unique positive   solution of
	\eqn{ 
	\label{b+}
	(1+w_c(B))(1+w_c(B)/(q-1)) = \left( \frac{d}{d-2} \right)^2.
	}
The following theorem describes the first-order phase transition of random cluster model:

\begin{theorem}[First-order phase transition for random-cluster model on regular graph]
\label{thm-|R|C-regular}
If $q>2$, the random cluster model on a sequence of locally tree-like $d$-regular graph undergoes a {\em first-order phase transition} at the critical curve $\mathscr{C}$ parameterised by the equation \eqref{crit-curve}. More precisely, for $ 0 \leq  B < B_+$,
	\eq{
	\partial_w \varphi (w_c(B)^+,B) \neq 	\partial_w \varphi (w_c(B)^-,B).
	}
where $\varphi(w,B)$ is defined in \eqref{thermodynamic-limit-|R|C-q>2-|R|RG} with $w=\e^\beta-1$. For $(w, B) \not \in \mathscr{C}$, instead, there is no phase transition, that is, $w\mapsto \varphi(w,B)$ is analytic.
\end{theorem}

\begin{remark}[How does $q \neq 2$ arise?]
\label{rem-first-order-q>2}
{\rm 
When $q\searrow 2$, the curve shrinks to the critical point of the Ising model $(w_c^{\rm \sss Ising},0)$ with $w_c^{\rm Ising}=\e^{\beta_c^{\rm \sss Ising}}-1$.}\hfill \ensymboldefinition
\end{remark}

\paragraph{\bf Relation to Basak, Dembo and Sly \cite{BasDemSly23} for Potts models with $q\geq 3$.} We close this section by explaining the relation to the work of Basak, Dembo and Sly in \cite{BasDemSly23} that studies the $q$-states Potts model on locally tree-like regular graphs with an external field. In this paper, \cite[Assumption 1.4]{BasDemSly23} makes an assumption about the form of the pressure per particle in this setting that roughly states that the quenched and annealed pressure per particle agree. More precisely,  \cite[Assumption 1.4]{BasDemSly23} states that the quenched pressure per particle is the same as the Bethe functional. Our results allow us to prove this assumption. We first show that the quenched and annealed pressure per particle are the same for {\em random} $d$-regular graphs, which will be an essential ingredient in the proof:

\begin{theorem}[Quenched equals annealed for random $d$-regular graphs]  
\label{thm-assumption1.4-BasDemSly23}
Let $G_n$ be a random $d$-regular graph with $d\geq 3$ and $n$ vertices. Then, the quenched and annealed pressures per particle are equal, i.e., for all $\beta \geq 0$ and $B \in \R$,
	\eqn{
	\lim_{n\rightarrow \infty} \frac{1}{n} \expec\big[\log Z_{G_n}(q,w,B)\big] = \lim_{n \rightarrow \infty} \frac{1}{n}\log \E[Z_{G_n}(q,w,B)]= \varphi(w,B).
	}
\end{theorem}
\medskip

Theorem \ref{thm-assumption1.4-BasDemSly23} has the following direct consequence, showing that the pressure per particle equals the Bethe functional:

\begin{corollary}[Bethe prediction holds for all $B\geq 0$]
Let $G_n$ be a regular graph on $n$ vertices and with degree $d\geq 3$ that converges locally to a $d$-regular graph. Then \cite[Assumption 1.4]{BasDemSly23} of  Potts models with $q\geq 3$ holds for all  $\beta \geq 0$, $B \geq 0$. 
\end{corollary}

\proof Let $G_n$ be a regular graph on $n$ vertices and with degree $d\geq 3$ that converges locally to a $d$-regular tree. The limit of the pressure per particle for Ising models exists by local convergence, as shown by Dembo, Montanari, Sly and Sun \cite{DemMonSlySun14}, and the limit of the pressure per particle does not depend on the precise sequence but only on the local limit being a regular tree. Thus, we only need to show that the limit is equal to the Bethe functional, as this is what \cite[Assumption 1.4]{BasDemSly23} states. Further, that same paper proves that the annealed pressure of the random $d$-regular graph equals the Bethe functional for all $B\geq 0$. By Theorem \ref{thm-assumption1.4-BasDemSly23}, the quenched and annealed pressure per particle agree for the random $d$-regular graph and all $B\geq 0$. This means that, for any $d$-regular graph $G_n$ that converges locally to the regular tree, the quenched pressure per particle converges to the Bethe functional.
\qed

\subsection{Pressure per particle for the extended Ising model with fixed-sign fields}
\label{sec-eIsing-pressure}
In this section, we investigate the pressure per particle for the extended Ising model with external fields that have a fixed sign, i.e., we consider
	\eqn{ \label{def-eising-rep}
	Z_{G}^{\sss \rm eIsing}(\beta,k,h):=	 \sum_{\bsigma\in \{-1,1\}^{V(G)}}\exp \Big( \beta \sum_{\{u,v\}\in E(G)} \sigma_u \sigma_v + \sum_{u \in V(G)} B_v \sigma_u \Big),
	}
where we define the vertex-dependent external fields $(B_v)_{v\in V(G)}$ by
	\eqn{
	B_v=k d_v+h.
	}
The following theorem shows that the pressure per particle of such an extended Ising model exists provided the external fields have fixed sign:

\begin{theorem}[Thermodynamic limit of extended Ising model]
\label{thm-termodynamic-limit-extended-Ising}
Let $(G_n=(V(G_n),E(G_n)))_{n\geq 1}$ be a sequence of locally tree-like random graph having $|V(G_n)|=n$ vertices. Suppose that $\beta \geq 0$, $k\geq 0$ and  $kd_{\min}+h > 0$, where
	\eqn{ \label{def-dmin}
	 d_{\min}= \inf_{n \geq 1} \min_{v \in V(G_n)} d_v
	}
is the minimal degree in the graph. Then, as $n \rightarrow \infty$, 
	\eqn{
	\label{thermodynamic-limit-|R|C-extended-Ising}
	\frac{1}{n} \log Z_{G_n}^{\rm \sss eIsing}(\beta,k,h)\convp  \varphi^{\rm \sss eIsing}(\beta,k,h).
	}
\end{theorem}
By symmetry of the spins, the same result holds when $k\leq 0$ and  $kd_{\min}+h< 0$. We refer to Theorem \ref{thm-pressure-per-particle-locally-tree-like} for a more detailed statement, in which a reprentation of $ \varphi^{\rm \sss eIsing}(\beta,k,h)$ is given in terms of message passing recursion relations. For $k\geq 0$, the condition $k d_{\min}+h > 0$ is verified either when $h > 0$ or when  $kd_{\min}$ is sufficiently large. While the first case  can be interpreted as the high external field regime, the second one corresponds to the low temperature regime. The following corollary investigates the high external field, and low temperature, regimes of the random cluster model on locally tree-like random graphs:

\begin{corollary}[High field or low temperature regimes]
Consider the random cluster model with parameters $q\geq 2$, $w=\e^\beta -1$ and external field $B \geq 0$ on a sequence of locally tree-like random graphs having $|V(G_n)|=n$ vertices.    
	\begin{description}
		\item [High external field regime] If $B > \log (q-1) $ then  
		\eqn{ \label{eih1}
		\lim_{n\rightarrow \infty} \frac{1}{n} \log Z_{G_n}(q,w,B)=\frac{\beta^*}{2}\expec[D_o]+\frac{1}{2} \log(\e^{-B}(q-1))
	+\varphi^{\rm \sss eIsing}(\beta^*,k,h).	
	}
\item[Low temperature regime] Suppose that 
\eq{
d_{\min} \geq 3, \qquad w > \frac{(q-1)^{2/d_{\min}}-1}{1-(q-1)^{2/d_{\min}-1}},
}
where $d_{\min}$ is the minimal degree defined in \eqref{def-dmin}. Then, \eqref{eih1} holds.
	\end{description}
\end{corollary}


Theorem \ref{thm-termodynamic-limit-extended-Ising} is proved in Section \ref{sec-termodynamic-limit-extended-Ising}. While we do not give the precise formula for $\varphi^{\rm \sss eIsing}(\beta,k,h)$ here, its form will follow from the proof (see Theorem \ref{thm-pressure-per-particle-locally-tree-like} below). The proof of Theorem \ref{thm-termodynamic-limit-extended-Ising} follows the proof in \cite{Hofs25}, which, in turn, is inspired by \cite{DemMon10a, DemMonSlySun14,DomGiaHof10}.

\subsection{Discussion and open problems}
\label{sec-discussion-open-problems}
In this section, we discuss our results and state open problems and conjectures. 
\medskip

\paragraph{\bf Relation to the literature for regular graphs.} Dembo, Montanari, Sly and Sun \cite{DemMonSlySun14} show that  \cite[Assumption 1.4]{BasDemSly23}  holds for $d$ even and all $B\geq 0$. Bencs, Borb\'enyi and Csikv\'ary \cite{BenBorCsi23} prove that it holds for all $d$ and $B=0$, improving upon an earlier result by Helmuth, Jenssen and Perkins \cite{HelJenPer23} that applies to $q$ sufficiently large. With Theorem \ref{thm-assumption1.4-BasDemSly23}, we now know that it holds for {\em all} $\beta>0$ and $B\geq 0$. The main results in \cite{BasDemSly23} are related to the existence of the limiting Potts measure on locally tree-like regular graphs in the local sense, the uniqueness of these limiting measures, and how these relate to boundary conditions. For the phase diagram of some values $q,d$, we refer to \cite[Figure 1]{BasDemSly23}. We also refer to \cite{BasDemSly23} for further consequences of \cite[Assumption 1.4]{BasDemSly23}. 
\medskip

Aside from the existence and shape of the pressure per particle, and the phase transitions that arise from their non-analyticities, there are deep connections to other problems that have received substantial attention in the literature in the past years. The {\em dynamics} of Potts models, and the metastable states that arise from the phase transitions, are investigated by Coja-Oghlan et al.\ in \cite{CojGalGolRavSteVig23}, to which we also refer for the history of the problem and appropriate references. Combinatorial aspects, including the relation to Tutte polynomials, are highlighted in \cite{BenBorCsi23} and the references therein. For the relation to the computational hardness of partition functions, we refer to \cite{GalSteVigYan16} and the references therein.
\medskip

\paragraph{\bf Thermodynamic limits of extended Ising models.} It would be highly natural, as well as interesting, to show that the partition function of the extended Ising model exists, even when not all external fields have the same sign. This is our next conjecture:

\begin{conjecture}[Pressure per particle of extended Ising model]
\label{conj-eising}
Let $G_n$ be a graph sequence that is locally tree-like. Let $Z_{G_n}^{\sss \rm eIsing}(\beta,k,h)$ denote the partition function of the extended Ising model with   parameters $\beta,k,h $ defined as in \eqref{def-eising}. Then, there exists $\varphi^{\sss \rm eIsing}(\beta,k,h)$ such that
	\eqn{
		\varphi_n^{\sss \rm eIsing}(\beta,k,h)=\frac{1}{n}\log Z_{G_n}^{\sss \rm eIsing}(\beta,k,h)
		\rightarrow \varphi^{\sss \rm eIsing}(\beta,k,h).
	}
\end{conjecture}
Conjecture \ref{conj-eising} has the following immediate corollary:

\begin{corollary}[Pressure per particle of random cluster model]
	\label{cor-pressure}
	Let $G_n$ be a graph sequence that is locally tree-like.  Assume that Conjecture \ref{conj-eising} holds for $G_n$. Then,
	\eqn{
	\frac{1}{n}\log Z_{G_n}(q,w,B)
		\rightarrow \varphi(q,w,B):= \frac{1}{2} \log (\e^{-B}(q-1)) +  \frac{1}{2}\beta^* \E[D_o] +  \varphi^{\sss \rm eIsing}(\beta^*,k,h),
	}
	where $\beta^*,k,h$ are given in \eqref{h-choice}--\eqref{beta-choice}.
\end{corollary}

\proof This immediately follows since $|E(G_n)|/n\rightarrow \expec[D_o]/2$.
\qed
\medskip

The difficulty in proving Conjecture \ref{conj-eising} is that the external fields in our representation in terms of the extended Ising model have both positive and negative contributions. Because of this, we cannot use the usual techniques for proving convergence of the pressure per particle. For example, it is not even clear what the {\em ground state}, corresponding to $\beta=\infty$ or zero temperature, of this model is. 
\medskip

\paragraph{\bf Negative external fields.}The Ising model, arising for $q=2$, is perfectly symmetric, so that $Z_{G}(q,w,-B)=Z_{G}(q,w,B)$, which allows us to translate results for $B\leq 0$ to $B\geq 0$. This property, however, does not extend to general $q$. Our results only apply to non-negative $B$. We next explore what happens for $B<0$ for locally tree-like regular graphs: 

\begin{remark}[Implications of lower bound for regular graphs]
\label{rem-regular-negative-B}
{\rm Let $G_n$ be a regular graph that is locally tree-like. Note that the rank-2 {\em lower bound} in Lemma \ref{lowerzq} is always true. We next look at consequences of this fact. Suppose that
	\eqn{
	\label{pressure-negative-B-a}
	\frac{1}{n}\log Z^{\sss (2)}_G(q,w,B)\convp \varphi_{\rm B}^{\rm \sss(e)}(w,B),
	}
as well as
	\eqn{
	\label{pressure-negative-B-b}
	\frac{1}{n}\log \expec[Z^{\sss (2)}_G(q,w,B)]\rightarrow \varphi_{\rm B}^{\rm \sss(e)}(w,B),
	}
so that quenched and annealed pressures agree for the rank-2 approximation. Here, we think of $\varphi_{\rm B}^{\rm \sss(e)}(w,B)$ as the Bethe 
functional of the (extended) Ising model, but this equality is not needed in what follows. Note that  the lower bound in terms of the 
Bethe functional follows from  Corollary \ref{cor-Bethe}.
\smallskip

Further, assume that we can also show the annealed statement that
	\eqn{
	\label{annealed-pressure-negative-B}
	\frac{1}{n}\log \expec[Z_G(q,w,B)]\rightarrow \varphi_{\rm B}(w,B),
	}
where the latter is the Bethe functional of the original random cluster measure.  Our final assumption is that the (hopefully reasonably explicit) formulas for 
$\varphi_{\rm B}(w,B)$ and $\varphi_{\rm B}^{\rm \sss(e)}(w,B)$ allow us to relate them as
	\eqn{
	\label{pressure-negative-B-c}
	\varphi_{\rm B}(w,B)=\varphi_{\rm B}^{\rm \sss(e)}(w,B),
	}
Of course, it may be that this is only true for part of the parameter choices of $(\beta,B)$.
\smallskip

Then we claim that the quenched and annealed pressure per particle for locally tree-like regular graphs converges to $\varphi_{\rm B}(w,B)=\varphi_{\rm B}^{\rm \sss(e)}(w,B)$. This may be used to extend our results to  $B<0$. Indeed, by Lemma \ref{lowerzq},
	\eqn{
	\frac{1}{n}\log Z_G(q,w,B)\geq \frac{1}{n}\log Z^{\sss (2)}_G(q,w,B) \convp \varphi_{\rm B}^{\rm \sss(e)}(w,B).
	}
Further, since the quenched pressure is bounded from above by the annealed one, also
	\eqn{
	\expec\Big[\frac{1}{n}\log Z_G(q,w,B)\Big]\leq \frac{1}{n}\log \expec[Z_G(q,w,B)]\rightarrow \varphi_{\rm B}
	(w,B)= \varphi_{\rm B}^{\rm \sss(e)}(w,B),
	}
where the last inequality follows from \eqref{pressure-negative-B-c}. Then, we conclude that
	\eqn{
	\frac{1}{n}\log Z_G(q,w,B)\convp \varphi_{\rm B}(w,B),
	\qquad \frac{1}{n}\log \expec[Z_G(q,w,B)]\rightarrow \varphi_{\rm B}(w,B),
	}
as desired.
\smallskip

Let us close by discussing what we know about these assumptions in \eqref{pressure-negative-B-a}--\eqref{pressure-negative-B-c} for random regular graphs. We note that \eqref{pressure-negative-B-a} and \eqref{pressure-negative-B-b} follow for random regular graphs, since quenched and annealed Ising models have the same pressure per particle (see e.g., \cite{Can19}). This also gives an explicit expression for $\varphi_{\rm B}^{\rm \sss(e)}(w,B)$. 

The computation in \eqref{annealed-pressure-negative-B} is for an {\em annealed} system, and should thus be simpler than for the {\em quenched} system. See e.g., \cite[Equation (13)]{GalSteVigYan16} for an explicit expression for the annealed partition function for $B=0$, for Potts models. Further, one may hope that the relatively explicit expressions for $\varphi_{\rm B}(w,B)$ and $\varphi_{\rm B}^{\rm \sss(e)}(w,B)$ can be compared to prove \eqref{pressure-negative-B-c}.
}\hfill \ensymboldefinition
\end{remark}
\medskip

\paragraph{\bf The nature of the phase transition in random cluster measures.} It would be highly interesting to use our representation to study the critical nature of the extended Ising model on locally tree-like random graphs beyond the {\em regular} setting. It can be expected that the phase transition is second order when $q\in(1,2]$, while it becomes first-order when $q>2$. For $q>2$, one can expect the phase transition to persist even for $B>0$, as follows from our results combined with those of \cite{BasDemSly23}, as well as for the annealed case in \cite{GiaGibHofJanMai25}. In the latter paper, special attention was given to settings where the random graphs have {\em power-law} degrees. Remarkably, the phase transition may depend on the power-law exponent for $q>2$. Indeed, it was shown that the phase transition is second-order when the power-law exponent is below $\tau(q)$ for some $\tau(q)\in (3,4)$, while it is first-order when $\tau\geq \tau(q)$. Here, $\tau>2$ is characterised by the limiting proportion of vertices of degree at least $k$ being of the order $k^{-(\tau-1)}$ in the large-graph limit. Does this `smoothing' of the phase transition also occur in the quenched case?
\medskip

\paragraph{\bf Connected components in random cluster measures.}
One particularly interesting aspect could be to consider the connected component of the random cluster measure, and show that this has a phase transition. So far, most phase transition results for the Ising or Potts model have been proved for the {\em magnetisation} instead.
\smallskip

In more detail, denote $\mu_{\beta,q}(x,y)=\mu_{p,q}(\sigma_x=\sigma_y)-1/q$ for the two-point correlation function of the Potts model. Then, (cf.\ \cite[Theorem (1.16)]{Grim06}) for $p=1-\e^{-\beta}$ and $q\in {\mathbb N}$,
	\eqn{
	\label{tau-|R|C-Potts}
	\mu_{\beta,q}(x,y)=(1-1/q)\phi_{p,q}(x\conn y).
	}
This allows one to compute Potts quantities using the random cluster representation. In particular, the critical value $\beta_c(q)$ should be the value for which $\phi_{p,q}(x\conn y)$ will become positive uniformly in $x$ and $y$, or for random $x$ and $y$, i.e., there is a {\em giant component} in the corresponding random cluster measure. It would be highly interesting to explore these connections further.

\section{Rank-2 approximations to random cluster models}
\label{sec-quenched-random-cluster}
In this section, we rely on Bencs, Borb\'enyi and Csikv\'ary \cite{BenBorCsi23} to give a representation of general random cluster measures on random graphs with few cycles.

\subsection{Rank-2 approximations to random cluster models without external fields}
We rely on \cite[Theorem 2.4]{BenBorCsi23}, which we restate here. For this, we first define some additional notation. For $q>0$ and $w>0$, we let
	\eqn{
	\label{ZG-sum-A-def}
	Z_G(q,w)=\sum_{A\subseteq E(G)} q^{k(A)}w^{|A|},
	}
where we recall that $k(A)$ denotes the number of connected components of the graph $(V(G),A)$. To relate this to the random cluster model (or even Potts models), we have to take $w=\e^{\beta}-1$. Indeed, comparing to \eqref{RC-partition}, we see that
	\eqn{
	w=\frac{p}{1-p}=\e^{\beta}-1,
	}
since $p=1-\e^{-\beta}$, so that
	\eqn{
	\label{RC-partition-rew}
	Z_{n,\rm RC}(p,q)=\sum_{\omega} q^{k(\omega)} \prod_{e\in E(G)} p^{\omega(e)}(1-p)^{1-\omega(e)}=(1-p)^{|E(G)|}Z_G(q,w).
	}
Thus, obviously, to study the partition function of the random cluster model, it suffices to study $Z_G(q,w)$ in \eqref{ZG-sum-A-def}. 
In terms of this notation, \cite[Theorem 2.4]{BenBorCsi23} reads as follows:

\begin{theorem}[Rank-2 approximation of random cluster partition function \protect{\cite[Theorem 2.4]{BenBorCsi23}}]
	\label{thm-BenBorCsi23}
	Fix $q\geq 2$, and let $G$ be a finite graph. Then
	\eqn{
		Z_G(q,w)\geq Z_G^{\sss(2)}(q,w),
	}
where 
\eq{
Z_G^{\sss(2)}(q,w) =\sum_{S \subseteq V(G)} (q-1)^{|V(G) \setminus S|}  (1+w)^{|E(S)|} (1+w/(q-1))^{|E(V(G)\setminus S)|}.
}
	Let $G$ be a graph with $L=L(G,g)$ cycles of length at most $g-1$ and size $|V(G)|=n$. Then also
	\eqn{
		Z_G(q,w)\leq q^{n/g+L}Z_G^{\sss(2)}(q,w).
	}
\end{theorem}
\smallskip

We refer to \cite{BenBorCsi23} for its proof.

\subsection{Rank-2 lower bound for random cluster models with external fields}
We next adapt the proof of Theorem \ref{thm-BenBorCsi23} to $B>0$, in anticipation of proving Theorem \ref{thm-B>0}. Many parts of this analysis are actually closely related, and we spell out some of the ingredients of the proof.
\medskip

Our  aim is to compare $Z_G(q,w,B)$ with $Z^{\sss (2)}_G(q,w,B)$, where we recall that
	\eq{
 	Z^{\sss (2)}_G(q,w,B)
	= \sum_{S\subseteq V(G)}(1+w)^{|E(S)|}\left(1+\frac{w}{q-1}\right)^{|E(V(G)\setminus S)|} ((q-1)\e^{-B})^{|V(G)|-|S|}.
	}
In the following lemma, we reduce $q$, and at the same time isolate the dependence on $B$:
\begin{lemma}[Reduction of $q$ and dependence on $B$] 
\label{zqrec}
For all $q$ and $B\in \mathbb{R}$,
	\eq{
Z_G(q,w,B) = \sum_{S \subseteq V(G)} \e^{B(|S|-|V(G)|)} (1+w)^{|E(S)|} Z_{G\setminus S} (q-1,w).	
}
\end{lemma}
\begin{proof} This proof is an adaptation of \cite[Proof of Lemma 2.2]{BenBorCsi23}. Recall the Potts model and Lemma \ref{pott=rc} saying that with $\beta = \log (1+w)$, if $q$ is an integer larger than $2$,
	\eqn{
	Z_G^{\sss \rm Potts}(q,\beta,B) = \e^{Bn} Z_G(q,w,B).	
	}
Note that the partition function of the Potts model can be represented as 
	\eqan{
	Z_G^{\sss \rm Potts}(\beta,B) &= \sum_{(S, S_1, \ldots, S_{q-1}) \in \mathcal{P}_q(V(G))} \e^{B|S|} (1+w)^{|E(S)|} \prod_{i=1}^{q-1} (1+w)^{|E(S_i)|},\nn
}	
where $\mathcal{P}_k(U)$ is the set of all $k$-partitions of a set $U$. Thus,
	\eqan{
	Z_G^{\sss \rm Potts}(q,\beta,B)
	&= \sum_{S \subseteq V(G)} \e^{B|S|} (1+w)^{|E(S)|}  \sum_{( S_1, \ldots, S_{q-1}) \in \mathcal{P}_{q-1}(V(G)\setminus S)}  \prod_{i=1}^{q-1} (1+w)^{|E(S_i)|}\\
	&= \sum_{S \subseteq V(G)} \e^{B|S|} (1+w)^{|E(S)|} Z_{G\setminus S}^{\sss \rm Potts}(q-1,\beta,0),\nn
	}	
Combining the last two display equations, we get for all integers $q \geq 2$, 
	\eq{
	Z_G(q,w,B) = \sum_{S \subseteq V(G)} \e^{B(|S|-|V(G)|)} (1+w)^{|E(S)|} Z_{G\setminus S} (q-1,w).	
	}
Moreover, for a finite graph $G$, both expressions in the above equation are polynomials in $q$. Therefore, the equation holds for all $q$. 
\end{proof}
The next lemma provides a rank-2 lower bound on the partition function:

\begin{lemma}[Rank-2 lower bound] 
\label{lowerzq}
For all $q \geq 2$ and $B \geq 0$,
	\eqn{
	Z_G(q,w,B) \geq Z^{\sss (2)}_G(q,w,B).
	}
\end{lemma}

\begin{proof} By \cite[Lemma 2.1]{BenBorCsi23}, for $q\geq 1$,
	\eqn{
	\label{rank-1-lower-bound}
	Z_G(q,w)  \geq  q^{|V(G)|} \left(1+ w/q \right)^{|E(G)|}.
	}
This can be seen by noting that, since $k(A)\geq |V(G)|-|A|$ for any collection of edges $A$,
	\eqn{
	\label{ZG-sum-A-LB}
	Z_G(q,w)=\sum_{A\subseteq E(G)} q^{k(A)}w^{|A|}\geq \sum_{A\subseteq E(G)} q^{|V(G)|-|A|}w^{|A|}=q^{|V(G)|} \left(1+ w/q \right)^{|E(G)|}.
	}
\smallskip

Using \eqref{rank-1-lower-bound} and  Lemma \ref{zqrec}, now for $q-1\geq 1$, so that $q\geq 2$,
	\eqa{
	Z_G(q,w,B) &= \sum_{S \subseteq V(G)} \e^{B(|S|-|V(G)|)} (1+w)^{|E(S)|} Z_{G\setminus S} (q-1,w)\\
	&\geq \sum_{S \subseteq V(G)} \e^{B(|S|-|V(G)|)} (1+w)^{|E(S)|} (q-1)^{|V(G)|-|S|}(1+w/(q-1))^{|E(V(G)\setminus S)|}\\
	&= 	Z^{\sss (2)}_G(q,w,B).
}
\end{proof}

\subsection{Rank-2 upper bound for random cluster models with external fields}
We next proceed with the upper bound, which is only valid for $B\geq 0$:
	
	\begin{lemma}[Rank-2 upper bound]  
	\label{upperzq}
	For all $q \geq 2$ and $B \geq 0$,
		\eq{
		Z_G(q,w,B) \leq q^{\kL(G)}	Z^{\sss (2)}_G(q,w,B),	
	}
			where 
			\eq{ 
				\kL(G)= \max_{A \subseteq E} |\{ \textrm{\rm connected components of $G_A=(V(G),A)$ containing a cycle}\}|.
			}
	\end{lemma}
	
If we compare the upper bound in Lemma \ref{upperzq} to the one in Theorem \ref{thm-BenBorCsi23}, we see that the power of $q$ is slightly different. It turns out that the bound in Lemma \ref{upperzq} allow us to also investigate the annealed setting (for which we have to do a large deviation type estimate on the number of components with cycles), whereas the bound in Theorem \ref{thm-BenBorCsi23} is not strong enough for that. We refer to Section \ref{sec-Potts} for details.
	
	\begin{proof}  Let  $A=\{e \in E(G)\colon \omega_e=1 \}\subseteq E(G)$, let $V_1,\dots ,V_r$ be the vertex sets of the connected components of the graph $G_A=(V(G),A)$, and let $A_1,\dots ,A_r$ be the corresponding subsets of $A$. If $V_i$ is an isolated vertex, then, by convention, $A_i=\emptyset$ is the empty set. We define 
		\[
		\mathcal{S}_A =\{i\colon V_i \textrm{ is a tree} \}, \qquad \mathcal{L}_A =\{i\colon  V_i \textrm{ contains a cycle}\}.
		\]
		Since $|\kL_A| \leq \kL(G)$,
		\eqan{
		\label{first-bound-rank-2}
		Z_G(q,w,B)&=\sum_{A \subseteq E(G)} w^{|A|} 
	\prod_{C\in\mathscr{C}(A)}
	(1+(q-1)\e^{-B|C|})\\
	& \leq q^{\kL(G)} \sum_{A \subseteq E(G)} w^{|A|} 
	\prod_{i\in \mathcal{S}_A }
	(1+(q-1)\e^{-B|V_i|}).\nn
	}
		We say that a vertex set $R$ is {\em compatible} with $A$ if $R =\cup_{i \in I} V_i$ with $I \subseteq \kS_A$ is the union of some tree components of $A$. Note that $R$ may be the empty set.  We denote this relation by $R\sim A$. Furthermore, let $A\llbracket R\rrbracket$ be the edges of $A$ induced by the vertex set $R$. Note that if $R\sim A$, then $A\llbracket R\rrbracket$ is a forest. On the other hand, there is no restriction on $A\llbracket V(G)\setminus R\rrbracket.$
		\smallskip
		
Let $k(R,A\llbracket R \rrbracket)$ be the number of connected components of  $(R,A\llbracket R \rrbracket)$. Since $\mathcal{S}_A$ is the set of  tree components in $(V(G), A)$,
		\eqan{
		\label{second-bound-rank-2}
		\prod_{i \in \mathcal{S}_A }
		(1+(q-1)\e^{-B|V_i|}) &= \sum_{I \subseteq \kS_A} (q-1)^{|I|} \e^{-B \sum_{i \in I} |V_i|}\\
		&=\sum_{R\sim A}(q-1)^{k(R,A\llbracket R \rrbracket)} \e^{-B|R|}.\nn
		}
Therefore,
		\begin{align}
		\label{third-bound-rank-2}
			Z_G(q,w,B)&\leq q^{\kL(G)}\sum_{A \subseteq E(G)}\sum_{R: R\sim A}(q-1)^{k(R,A\llbracket R\rrbracket)}w^{|A|} \e^{-B|R|}\\
			&=q^{\kL(G)}\sum_{R \subseteq V(G)} \e^{-B|R|}\sum_{A\colon R\sim A} (q-1)^{k(R,A\llbracket R\rrbracket)}w^{|A\llbracket R\rrbracket|}\times w^{|A\llbracket V(G)\setminus R\rrbracket|}\nn\\
			&=q^{\kL(G)}\sum_{R \subseteq V(G)} \e^{-B|R|} (1+w)^{|E(V\setminus R)|}\sum_{D}(q-1)^{k(R,D)}w^{|D|},\nn
		\end{align}
		where, in the last sum, $D=A\llbracket R\rrbracket$ is a subset of the edges on the subgraph $(R,E(R))$ that is such that $(R,D)$ is a forest, and $k(R,D)$ the number of connected components of the graph on $R$ with edges given by $D$.  Then
		\eqan{
		\label{fourth-bound-rank-2}
			\sum_{D}(q-1)^{k(R,D)}w^{|D|}&=\sum_{D}(q-1)^{|R|-|D|}w^{|D|}= (q-1)^{|R|} \sum_{D} (w/(q-1))^{|D|}\\
			&\leq (q-1)^{|R|} (1+w/(q-1))^{|E(R)|}.\nn
		}
		Combining the last two estimates, and setting $S=V(G)\setminus R$, we obtain the desired result.
	\end{proof}
Now we are ready to complete the proof of Theorem \ref{thm-B>0}:

\begin{proof}[Proof of Theorem \ref{thm-B>0}] The proof of Theorem \ref{thm-B>0}(i), or the comparison of the partition functions $Z_G(q,w,B) $ and $ Z^{\sss (2)}_G(q,w,B)$, follows from Lemmas \ref{lowerzq} and \ref{upperzq}.  For (ii) or the estimate of $\kL(G)$,  we say that a connected component of $(V(G),A)$ is  {\em large} when it contains a cycle of length at least $k$, and small  otherwise. It is clear that there are at most $|V(G)|/k$ large cyclic components. On the other hand, the number of small cyclic components  is smaller than the maximal number of disjoint cycles of length at most $k-1$, and thus is smaller than $\sum_{i=2}^{k-1}\kL_i(G)$. 
\end{proof}

\section{Random cluster model on random regular graphs}
\label{sec-Potts}
In this section, we prove the main results for the random cluster model on locally tree-like regular graphs. This section is organised as follows. In Section \ref{sec-assumption1.4-BasDemSly23}, we investigate the quenched and annealed random cluster measures on random regular graphs, and prove Theorem \ref{thm-assumption1.4-BasDemSly23}. In Section \ref{sec-|R|C-regular}, we investigate the nature of the phase transition, and prove Theorem \ref{thm-|R|C-regular}.

\subsection{Quenched equals annealed: Proof of Theorem \ref{thm-assumption1.4-BasDemSly23}}
\label{sec-assumption1.4-BasDemSly23}
In this section, we prove Theorem \ref{thm-assumption1.4-BasDemSly23} using the relation to the degree-regular configuration model. 

\begin{proof}[Proof of Theorem \ref{thm-assumption1.4-BasDemSly23}]
Combining Theorem \ref{thm-B>0}, Corollary \ref{thm-gen-Ising},  we obtain
	\eqn{ \label{compzg}
	\e^{\lambda n}  Z_{G_n}^{\sss \rm Ising}(\beta^*,kd+h) \leq Z_{G_n}(q,w,B) \leq \e^{\lambda n} \e^{\kL(G_n)}  Z_{G_n}^{\sss \rm Ising}(\beta^*,kd+h), 
	}
where 
	\eq{
	\lambda = \frac{1}{2}\beta^*d + \frac{1}{2} \log (\e^{-B}(q-1)), \qquad h= 	\frac{1}{2} \log (\e^{B}/(q-1)),
	}
and we recall that 
	\eq{
	\kL(G_n)= \max_{A \subseteq E(G_n)} |\{ \textrm{\rm connected components of $G_A=(V(G_n),A)$ containing a cycle}\}|.
	}
We claim that 
	\eqn{
	\label{blg}
	\pp(\kL(G) \geq 2 n/k_n) \leq \exp(-nk_n/2), \qquad k_n =\lfloor (\log n)^{1/4} \rfloor.
	}
Using \eqref{compzg} and \eqref{blg} then gives
	\eqn{
	\label{qzqb}
	\lim_{n \rightarrow \infty}	\frac{1}{n} \log   Z_{G_n}(q,w,B) = \lambda+  \lim_{n \rightarrow \infty} 	\frac{1}{n} \log  Z_{G_n}^{\sss \rm Ising}(\beta^*,kd+h). 
	}
\smallskip

It has been shown in \cite[Theorem 1]{DemMonSlySun14} or \cite[Proposition 3.2]{Can19} that the quenched and annealed Ising pressure per particle are equal, i.e., 
	\eqn{
	\label{qir}
	\lim_{n \rightarrow \infty} 	\frac{1}{n} \log   Z_{G_n}^{\sss \rm Ising}(\beta^*,kd+h) = \lim_{n \rightarrow \infty} 	\frac{1}{n}  \log \E[Z_{G_n}^{\sss \rm Ising}(\beta^*,kd+h)].
	}
On the other hand, using \eqref{blg} and the fact that  
	\eqn{
	Z_{G_n}^{\sss \rm Ising}(\beta^*,kd+h) \leq \exp(Kn),
	}
where $K=K(q,d,w,B)$ is a constant, we obtain
	\eqa{ 
	\E[q^{\kL(G_n)}  Z_{G_n}^{\sss \rm Ising}(\beta^*,kd+h)] &\leq q^{4 n/ k_n}\E[  Z_{G_n}^{\sss \rm Ising}(\beta^*,kd+h)] +  q^{2Kn} \pp(\kL(G_n) \geq 2n/k_n) \notag \\
	& \leq 2q^{4 n/ k_n}\E[  Z_{G_n}^{\sss \rm Ising}(\beta^*,kd+h)].
	}
This estimate, together with \eqref{compzg},  implies that
	\eqn{ \label{ezqb}
	\lim_{n \rightarrow \infty}	\frac{1}{n} \log  \E[ Z_{G_n}(q,w,B)] = \lambda+  \lim_{n \rightarrow \infty} 	\frac{1}{n} \log  \E[ Z_{G_n}^{\sss \rm Ising}(\beta^*,kd+h)]. 
	}
It follows from \eqref{qzqb}, \eqref{qir}, and \eqref{ezqb}  that
		\eq{
			\lim_{n \rightarrow \infty}	\frac{1}{n} \log   Z_{G_n}(q,w,B)=\lim_{n \rightarrow \infty}	\frac{1}{n} \log  \E[ Z_{G_n}(q,w,B)]. 
		}
Considering the Potts models, i.e., letting $q\geq 2$ be an integer, \cite[Theorem 3]{DemMonSlySun14} states that
		\eq{
			\lim_{n \rightarrow \infty}	\frac{1}{n} \log  \E[ Z_{G_n}(q,w,B)] = \Phi^\star(w,B),
		}
		where $\Phi^\star(w,B)$ is  the so-called {\em Bethe free energy}.  Therefore, the last two equations imply
		\eq{
			\lim_{n \rightarrow \infty}	\frac{1}{n} \log   Z_{G_n}(q,w,B) = \Phi^\star(w,B),
		}
which is indeed \cite[Assumption 1.4]{BasDemSly23}. Thus, it remains to prove \eqref{blg}. 
\medskip

\paragraph{\bf Proof of \eqref{blg}.} Observe that
	\eqn{ \label{lgli}
	\kL(G) \leq \frac{n}{k_n} + \sum_{i=2}^{k_n-1} \kL_i(G),
	}
where 
	\[
	\kL_i(G) = \max\{l\colon G \textrm{ contains $l$ vertex-disjoint cycles of length $i$} \}. 
	\]
Recall that we call a  cyclic connected  components of $G_A=(V(G),A)$ for $A\subseteq E(G)$  {\em large} if it contains  a cycle of length at least $k_n$, and small  otherwise. It clear that there are at most $n/k_n$ large cyclic components. On the other hand, the number of small cyclic components  is at most the maximal number of disjoint cycles of length at most $k_n-1$, and thus is at most $\sum_{i=2}^{k_n-1}\kL_i(G)$. It suffices to show that 	
	\eqn{
	\label{klig}
	\max_{2 \leq i \leq k_n-1}\pp\big( \kL_i(G_n) \geq n/k_n^2\big) \leq \exp(-n k_n).
	}
Indeed,  the desired result \eqref{blg} then directly follows from \eqref{lgli} and \eqref{klig} and a union bound. 
\medskip

To prove \eqref{klig}, we fix  $i \leq k_n$, set $m=\lfloor n/k_n^2 \rfloor$ and consider
	\eqn{
	\label{ligna}
	\pp(\kL_i(G_n) \geq m) \leq \sum_{S\subseteq V(G_n)\colon |S|=im} \pp(G_n(S)  \textrm{ has $m$ disjoint $i$ cycles}),
	}
where $G_n(S)$ is the induced subgraph of $G_n$ with vertex set $S$. We call a cycle an {\em $i$ cycle} when it has length $i$. Given $S \subseteq V(G_n)$ with $|S|=im$, we pick a vertex $v_1 \in S$ arbitrarily and observe that
		\eqan{
		\label{recursion-prel}
		&\pp(G_n(S)  \textrm{ has $m$ disjoint $i$ cycles})\nn\\
		& \leq \sum_{\substack{S_1 \subseteq S \setminus \{v_1\} \\ |S_1|=i-1}} \pp\Big(G_n(\bar{S}_1)  \textrm{ contains an $i$  cycle, and~}G_n(S\setminus\bar{S}_1)  \textrm{ has $m-1$ disjoint $i$ cycles}\Big)\nn\\
		& \leq \sum_{\substack{S_1 \subseteq S \setminus \{v_1\} \\ |S_1|=i-1}} \frac{(i-1)!\prod_{v \in \bar{B}_1} d_v}{\prod_{j=0}^{i-1}(\ell_n-2j+1)} 
		\times \pp\Big(G^{\sss(1)}_n(S\setminus\bar{S}_1)  \textrm{ has $m-1$ disjoint $i$ cycles}\Big),
		}
where 
		\[
		\bar{S}_1=S_1 \cup \{v_1\}, \quad \ell_n =\sum_{v \in V(G_n)} d_v =nd,
		\]
and $G^{\sss(1)}_n(S\setminus\bar{S}_1)$ is the configuration model on $[n]$ with degrees $((d_v-2)_{v \in \bar{S}_1}, (d_v)_{v \not \in \bar{S}_1}).$
\medskip

Applying \eqref{recursion-prel} recursively, we arrive at 
		\eqan{ \label{pgnai-a}
		\pp(G_n(S)  \textrm{ has $m$ disjoint $i$ cycles})  
		&\leq  \sum_{S_1, \ldots,S_m} \, \, \prod_{s=1}^m  \frac{(i-1)!\prod_{v \in \bar{S}_s}  d_v}{\prod_{j=0}^{i-1}(\ell_n-2j+1-2(s-1)i)}\nn\\
		&\leq  \sum_{S_1, \ldots,S_m} \, \, \prod_{s=1}^m  \frac{(i-1)!\prod_{v \in S} d_v}{\prod_{j=0}^{i-1}(\ell_n-2j+1-2(s-1)i)},
		}
where the sum is taken over all the possible tubes of disjoint sets $(S_s)_{s=1}^m \subseteq S$ such that $|S_s|=i-1$ for all $1 \leq s\leq m$, and we also used that 
	\[
	\prod_{s=1}^m \prod_{v\in \bar{S}_s}d_v =\prod_{v \in S} d_v.
	\] 
Note that there are at most 
		\[
		\binom{im}{i-1}^m \leq \frac{(im)^{(i-1)m}}{((i-1)!)^m}
		\]
such choices for $(S_s)_{s=1}^m$. Thus,
		\eqan{ \label{pgnai}
		\pp(G_n(S)  \textrm{ has $m$ disjoint $i$ cycles})  
		&\leq  \sum_{S_1, \ldots,S_m} \, \,  ((i-1)!)^m  \frac{\prod_{v \in S}  d_v}{(\ell_n/2)^{im}} \notag \\
		& \leq (im)^{(i-1)m} \frac{\prod_{v \in V(G_n)}d_v}{(\ell_n/2)^{im}} \leq  \frac{\prod_{v \in V(G_n)}  d_v}{(\ell_n/2)^{m}},
		}
where we also used that $\ell_n-2mi \geq \ell_n/2$ since $m i = o(n)$. Combining \eqref{ligna} and \eqref{pgnai} gives 
		\eqa{
		\pp(\kL_i(G_n) \geq m) \leq 2^n\exp \left( \sum_{v \in V(G_n)} \log d_v -m \log (\ell_n/2)  \right)   \leq \exp(-nk_n),
		}
since $m=\lfloor n/k_n^2 \rfloor$ and $k_n= \lfloor (\log n)^{1/4} \rfloor$, and $\sum_{v \in V(G_n)} \log d_v  = n \log  d$. This completes the proof of \eqref{blg}, and thus of Theorem \ref{thm-assumption1.4-BasDemSly23}.
\end{proof}

\subsection{The nature of the phase transition: Proof of Theorem \ref{thm-|R|C-regular}} 
\label{sec-|R|C-regular}
We first recall the form of the free energy of the Ising model on a random $d$-regular graph. We define (see \cite{Can19})
	\eqn{
	\label{L-beta-def}
	L(\beta,t)=t\log(1/t)+(1-t)\log(1/(1-t))+d\log{F_{\e^{-2\beta}}(t)},
	}
where
	\eqn{
	\label{F-f-ann-Ising-|R|GG-rep}
	F_b(t)=\int_0^{t\wedge (1-t)} \log{f_b(s)}ds,
	\quad
	\text{with}
	\quad
	f_{b}(s)=\frac{b(1-2s) +\sqrt{1+(b^2-1)(1-2s)^2}}{2(1-s)}.
	}
Then, by \cite[Theorem 1.1]{Can19},
	\eqn{
	\varphi^{\sss \rm Ising}(\beta,z)
	=\frac{\beta d}{2}+\sup_{t\in[0,1]} G(\beta,z,t), \qquad  G(\beta,z,t)= L(\beta,t)+z(2t-1).
	}
We summarise here some properties of the maximiser $t_{\beta,z}$ of the function $G$:
\begin{itemize}
	\item [(a)] If $z \neq 0$ then $t_{\beta,z}$ is unique. Moreover, the function $t_{\beta,z}$ is differentiable w.r.t.\ $(\beta, z)$ if $z \neq 0$. 
	\item [(b)] If $z>0$ then $t_{\beta,z}>1/2$ and $t_{\beta,-z}=1-t_{\beta,z}<1/2$. Moreover, 
	\eq{
t_{\beta,0^+} = 1- t_{\beta,0^-}  >1/2 \quad \textrm{if } \quad \beta >\beta^{\sss \rm Ising}_c; \qquad 	t_{\beta,0^+}= t_{\beta,0^-}   =1/2 \quad \textrm{if } \quad \beta \leq  \beta^{\sss \rm Ising}_c,
}
where 
 \eq{
  \beta^{\sss \rm Ising}_c = \frac{1}{2} \log \left( \frac{d}{d-2} \right).
}
\item [(c)] $\beta\mapsto \varphi^{\sss \rm Ising}(\beta,z)$ is differentiable in $\beta$ for all $(\beta,z)$.  On the other hand,  $\varphi^{\sss \rm Ising}(\beta,z)$ is differentiable in $z$ for all $(\beta,z)$ except for  $\beta > \beta_c^{\sss \rm Ising}$ and $z=0$:   
\eq{
\partial_z 	\varphi^{\sss \rm Ising}(\beta,0^+) =- \partial_z 	\varphi^{\sss \rm Ising}(\beta,0^-) >0.
} 
\end{itemize}
We remark that 
	\eqn{
	\varphi(q,w,B)= \lambda  + \varphi^{\sss \rm Ising}(\beta^*, z),
	}
where 
	\eq{
	\beta^*= \frac{1}{4} \log ((1+w)(1+w/(q-1))), \quad z=kd +h,\qquad 
	k=\frac{1}{4}\log \left(\frac{1+w}{1+w/(q-1)}\right), 
	}
and
	\eq{
	\label{lambda-h-def}
	\lambda = \frac{1}{2}\beta^*d + \frac{1}{2} \log (\e^{-B}(q-1)),  \qquad h=\frac{1}{2} \log (\e^B/(q-1)). 	
	}
\smallskip	
	
By observation (c), $z=0$ is  the only possible value of $z$ for which the derivatives of $\varphi$ might be discontinuous.  Solving $z=kd +h=0$, we obtain 
\eqn{ 
\label{wcb-first}
	w= w_c(B):= \ell_{q,d}(\e^{-2h}), \qquad  \ell_{q,d}(x)=\frac{x^{2/d}-1}{1-\tfrac{1}{q-1}x^{2/d}}. 
}
Also by observation (c), while $\partial_{\beta^*} \varphi$ is continuous in the whole regime,  $\partial_z \varphi$ is discontinuous when  $\beta^* > \beta_c^{\sss \rm Ising}$  and $z=0$, and  $\partial_z \varphi$ is continuous otherwise. 
\smallskip

Next, we aim to check the condition $\beta^*(w_c(B)) > \beta_c^{\sss \rm Ising}$. By a direct computation, for $q>2$,
	\eqn{
	\label{wbp}
	w_c'(B) = \frac{2(q-1)\e^{-4h/d}}{d(q-1-\e^{-4h/d})^2} (2-q) \partial_B h = \frac{(q-1)(2-q)\e^{-4h/d}}{d(q-1-\e^{-4h/d})^2}<0.
	}
Consequently, 
	\eqn{ 
	\label{wb0q}
	B\mapsto w_c(B) \textrm{ is strictly decreasing when $B>0$}.
	}
We claim that, for $q > 2$,
	\eqn{ 
	\label{bw0}
		w_c(0)>0, \qquad  g(w_c(0)) > \left( \frac{d}{d-2} \right)^2. 
	}
where 
	\eqn{
	\label{g-w-def}
		g(w)=\e^{4\beta^*(w)}= (1+w)(1+w/(q-1)).
	}
We defer the proof of \eqref{bw0} to the end of the proof. Since $w_c(0)>0$ and $w_c(\cdot)$ is strictly decreasing by  \eqref{wbp}, there exists a unique $B^+ >0$ such that 
	\eq{
		w_c(B^+)=0, \qquad w_c(B) > 0 \quad \textrm{ if and only if } 0 \leq B<B^+.
	}
Since  $g(\cdot)$ is increasing in $\R_+$, $g(w_c(B))$ is  decreasing in $0 \leq B<B^+$. Moreover, $g(w_c(B^+))=g(0)=1$ and $g(w_c(0)) > (\tfrac{d}{d-2})^2$. Therefore, there exists   $B_+ \in (0,B^+)$ such that 
	\eq{
	g(w_c(B_+))= \left(\frac{d}{d-2} \right)^2, \qquad 	g(w_c(B))>  \left(\frac{d}{d-2} \right)^2  \quad \textrm{ iff } 0 \leq B<B_+.
	}
Particularly, if $ 0 \leq B<B_+$,
	\eq{
		\beta^*(w_c(B)) = \frac{1}{4} \log g(w_c(B)) > \frac{1}{2} \log \left( \frac{d}{d-2} \right) = \beta_c^{\ss \rm Ising},
	}
and $\beta^*(w_c(B)) < \beta_c^{\ss \rm Ising}$ otherwise. 
	
In conclusion, if $q > 2 $ and $(w,B) \in \mathscr{C}$, where
	\eqn{
		\label{critcurve}
		\mathscr{C}=\{(w,B)\colon w=w_c(B)~\text{\rm for some } 0 \leq   B < B_+\},
	}
then the derivative $\partial_w \varphi$ is discontinuous, i.e., $\partial_w \varphi (w_c(B)^+,B) \neq \partial_w \varphi (w_c(B)^-,B)$, i.e., we have a first-order phase transition. 
We conclude that it remains to prove \eqref{bw0}. 
\medskip

\paragraph{\bf Proof of \eqref{bw0}.} Setting $x=q-1>0$ and $a=2/d \in (0,1)$, we have 
	\eq{
		w_c(0)=\ell_{q,d}(q-1) = \frac{x^a-1}{1-x^{a-1}}>0,
	}
since $x>1$ for $q>2$. By \eqref{g-w-def}, $g(w)=(1+w)(1+w/x)$, and substituting $w_c(0)=(x^a-1)/(1-x^{a-1})=x(x^a-1)/(x-x^{a})$, 
we obtain 
	\eqn{
	g(w_c(0))=\Big(1+\frac{x(x^a-1)}{x-x^{a}}\Big)\Big(1+\frac{x^a-1}{x-x^{a}}\Big)
	=\frac{x^a(x-1)^2}{(x-x^a)^2}.
	}
Thus,
	\eqan{ \label{gwc0d}
		&g(w_c(0)) - \left( \frac{d}{d-2} \right)^2 \nn \\
		&= \frac{x^a(x-1)^2}{(x-x^a)^2} - \frac{1}{(1-a)^2}= \frac{1}{(1-a)^2(x-x^a)^2}[x^a(x-1)^2(1-a)^2-(x-x^a)^2] \nn\\
		&=\frac{x^{a/2}}{(1-a)^2(x-x^a)^2} \left(x^{a/2}(x-1)(1-a)+x-x^a\right) \left((x-1)(1-a) - x^{1-a/2}+x^{a/2} \right) \nn \\
		&=:\frac{x^{a/2}}{(1-a)^2(x-x^a)^2} A_1(x) A_2(x).
	}
Since $a=2/d \in (0,1)$,
	\eqn{ \label{sa1x}
	A_1(x)>0 \quad \textrm{ if } \quad x>1.
	}
	Moreover, 
	\eqa{
		A_2'(x) &= 1-a - (1-a/2)x^{-a/2} +  a x^{a/2-1}/2\\
		& = 1-a - x^{-a/2} + a(x^{-a/2} +   x^{a/2-1})/2 \\
		& \geq 1 - a - x^{-a/2} + a x^{-1/2}\geq 0,
	}
where, for the third line, we have used the Cauchy-Schwarz inequality in the form $u+v\geq 2\sqrt{uv}$, and the inequality that $1-a+ay-y^a \geq 0$ if $y>0$ and $a \in (0,1)$. In addition, $A_2(1)=0$. Therefore,
	\eq{
		A_2(x)>0 \quad \textrm{ if } \quad x>1.
	}
Since $x=q-1>1$ as $q>2$, this, together with \eqref{gwc0d} and \eqref{sa1x}, implies the desired result \eqref{bw0}.\qed

\section{Thermodynamic limit of extended Ising model: Proof of Theorem \ref{thm-termodynamic-limit-extended-Ising}}
\label{sec-termodynamic-limit-extended-Ising}
In this section, we give an overview of the proof of Theorem \ref{thm-termodynamic-limit-extended-Ising}. For fixed-sign parameters, this proof is a minor adaptation of the proof for the Ising model with constant external field. In fact, some of the crucial ingredients have been proved for general vertex-dependent external fields, so that this proof requires only minor modifications. Let 
	\eqn{
	\varphi^{\rm \sss eIsing}_n(\beta,k,h)=\frac{1}{n} \log Z_{G_n}^{\rm \sss eIsing}(\beta,k,h).
	}
In order to state the limit of the pressure per particle, we define the necessary quantities related to {\em message passing algorithms} and belief propagation, as already briefly discussed in Section \ref{sec-Bethe}.
\medskip

\paragraph{\bf Message passing quantities} Let $\bB=(B_v)_{v\in V(\Tree)}=(kd_v+h)_{v\in V(\Tree)}.$ Further, let $\Tree=(V(\Tree),E(\Tree))$ be the rooted random tree that arises as the local limit of our graph sequence. Fix such a rooted tree $\Tree$ with root $\vertex\in V(\Tree)$. Let $\beta\geq 0$ and $B_v=kd_v+h\geq B_{\min}>0$ for every $v\in V(\Tree)$. Let $\{x,y\}$ be an edge in $\Tree$. Then, let $\Tree_{x\rightarrow y}$ be the sub-tree of $\Tree$ rooted at $x$ which results from deleting the edge $\{x,y\}$ from $\Tree$. We write $(x,y)$ for the directed version of the edge $\{x,y\}$ that points from $x$ to $y$. 
\smallskip

Let $\sigma_x\mapsto \mu_{\Tree, x\rightarrow y}^{\sss\beta,\bB}(\sigma_x)$ be the marginal law of $\sigma_x$ for the Ising model on $\Tree_{x\rightarrow y}$ with inverse temperature $\beta$ and vertex-dependent external fields $\bB=(B_v)_{v\in V(\Tree)}$. We will mostly be interested in $(x,y)=(\vertex,j)$ or $(x,y)=(j,\vertex)$ for some child $j$ of $\vertex$. We emphasise that $\mu_{\Tree, x\rightarrow y}^{\sss\beta,\bB}(\sigma_x)$ are random variables when $\Tree$ is a random tree. 
\smallskip

\invisible{Below, we rely on two relations between the $\mu_{\Tree, x\rightarrow y}^{\sss\beta,\bB}(\sigma_x)$. The first is that the law $\mu_{\Tree, \vertex}^{\sss\beta,\bB}(\sigma)$ of the spin at the root of the tree can be obtained from $\mu_{\Tree, j\rightarrow \vertex}(\sigma_j)$ for all $j\in \partial \vertex$ as, for $\sigma\in \{+,-\}$,
	\eqn{
	\label{root-spin-tree-distribution}
	\mu_{\Tree, \vertex}^{\sss\beta,\bB}(\sigma)=\frac{\e^{B_\vertex\sigma} 
	\prod_{j\in \partial \vertex} \Big(\sum_{\sigma_j} \e^{\beta \sigma \sigma_j} \mu_{\Tree, j\rightarrow \vertex}^{\sss\beta,\bB}(\sigma_j)\Big)}{\sum_\sigma \e^{B_\vertex \sigma} 
	\prod_{j\in \partial \vertex} \Big(\sum_{\sigma_j} \e^{\beta \sigma \sigma_j} \mu_{\Tree, j\rightarrow \vertex}^{\sss\beta,\bB}(\sigma_j)\Big)}.
	}
Further, by the tree recursion and the fact that the Ising models on the subtrees at the root are independent, for $\sigma\in \{+,-\}$,
	\eqn{
	\label{tree-recursion-massage-passing}
	\mu_{\Tree, \vertex \rightarrow j}^{\sss\beta,\bB}(\sigma)
	=\e^{B_\vertex\sigma}\prod_{j'\in \partial \vertex\colon j'\neq j} \Big(\sum_{\sigma_{j'}} \e^{\beta \sigma \sigma_{j'}} \mu_{\Tree, j'\rightarrow \vertex}^{\sss\beta,\bB}(\sigma_j)\Big).
	}
Equations \eqref{root-spin-tree-distribution} and \eqref{tree-recursion-massage-passing} can be thought of as {\em messages} or {\em beliefs} that are being passed between vertices and their neighbours. This explains the terminology of {\em message passing}, or {\em belief propagation}.
\smallskip}

We let $\partial \vertex$ be the children of $\vertex\in V(\Tree)$, and define, for $\bB=(B_v)_{v\in V(\Tree)}=(kd_v+h)_{v\in V(\Tree)},$
	\eqn{
	\label{Phi-Tree-vertex}
	\Phi_{\Tree}^{\rm vx}(\beta,\bB)=\log \Big[\sum_\sigma \e^{B_\vertex\sigma} 
	\prod_{j\in \partial \vertex} \Big(\sum_{\sigma_j} \e^{\beta \sigma \sigma_j} \mu_{\Tree, j\rightarrow \vertex}^{\sss\beta,\bB}(\sigma_j)\Big)\Big],
	}
and 
	\eqn{
	\label{Phi-Tree-edge}
	\Phi_{\Tree}^{\rm e}(\beta,\bB)
	=\frac{1}{2} \sum_{j\in \partial \vertex} \log \Big[\sum_{\sigma, \sigma_j} 
	\e^{\beta \sigma \sigma_j} \mu_{\Tree, j\rightarrow \vertex}^{\sss\beta,\bB}(\sigma_j)\mu_{\Tree, \vertex\rightarrow j}^{\sss\beta,\bB}(\sigma)\Big].
	}
Our main result is then the following:

\begin{theorem}[Pressure per particle on locally tree-like graphs]
\label{thm-pressure-per-particle-locally-tree-like}
Consider a graph sequence $(G_n)_{n\geq1}$ of size $|V(G_n)|=n$ that converges locally in probability to the random rooted tree $\Tree$. Then, for every $\beta\geq 0$ and $B_v=kd_v+h\geq B_{\min}>0$ for every $v\in V(G_n)$,
	\eqn{
	\frac{1}{n}\log Z_{G_n}^{\rm \sss eIsing}(\beta,k,h)\convp \varphi^{\rm \sss eIsing}_n(\beta,k,h),
	}
where 
	\eqn{
	\label{varphi-general-formula}
	\varphi^{\rm \sss eIsing}_n(\beta,k,h)=\expec[\Phi_{\Tree}^{\rm vx}(\beta,\bB)-\Phi_{\Tree}^{\rm e}(\beta,\bB)\Big],
	}
and the expectation is w.r.t.\ the random rooted tree $\Tree$.
\end{theorem}
\smallskip
	
Theorem \ref{thm-pressure-per-particle-locally-tree-like} obviously implies Theorem \ref{thm-termodynamic-limit-extended-Ising}, and gives an explicit formula for the limiting pressure per particle $\varphi(\beta,B)$. We next discuss its proof, for which, rather than studying  $\varphi^{\rm \sss eIsing}_n(\beta,k,h)$ directly, we study its derivative w.r.t.\ $h$. In more detail, we use that
	\eqn{
	\frac{\partial}{\partial h}\varphi^{\rm \sss eIsing}_n(\beta,k,h)
	=\frac{1}{n}\sum_{u\in V(G_n)} \langle\sigma_u\rangle_{\mu_n}
	=\expec\Big[ \langle\sigma_{\vertex_n} \rangle_{\mu_n}\mid G_n\Big],
	}
where $\mu_n$ is the extended Ising measure on $G_n$, i.e.,
	\eqn{
	\mu_n(\bsigma)=\frac{1}{Z_{G_n}^{\rm \sss eIsing}(\beta,k,h)}\exp \Big( \beta \sum_{\{u,v\}\in E(G_n)} \sigma_u \sigma_v + \sum_{u \in V(G_n)} B_v \sigma_u \Big),
	}
$\langle\cdot \rangle_{\mu_n}$ denotes the expectation w.r.t.\ $\mu_n$, and $\vertex_n\in V(G_n)$ is chosen uniformly at random. By the GKS inequality (see, e.g., \cite[Lemma 10.1]{Hofs25}), we can bound, for every $t\geq 0$, and using that $B_v=kd_v+h\geq B_{\min}>0$ for every $v\in V(G_n)$,
	\eqn{
	\label{local-weak-Ising-CM}
	 \langle\sigma_{\vertex_n} \rangle^f_{B_{\vertex_n}^{\sss(G_n)}(t)} \leq  \langle\sigma_{\vertex_n} \rangle_{\mu_n} 
	\leq  \langle\sigma_{\vertex_n} \rangle^+_{B_{\vertex_n}^{\sss(G_n)}(t)},
	}
where $B_{\vertex_n}^{\sss(G_n)}(t)$ is the $t$-neighbourhood of $\vertex_n\in V(G_n)$ in $G_n$, and $\langle\cdot \rangle^{+/f}$ denote the expectation w.r.t.\ the extended Ising model with + and free boundary conditions, respectively. 

We now rely on local convergence to obtain that 
	\eqn{
	\langle\sigma_{\vertex_n} \rangle^+_{B_{\vertex_n}^{\sss(G_n)}(t)}\convp \expec\big[\langle\sigma_{\vertex} \rangle^+_{B_{\vertex}^{\sss(G)}(t)}\big],
	}
where $B_{\vertex}^{\sss(G)}(t)$ is the $t$-neighbourhood of $\vertex\in V(\Tree)$ in the limiting rooted tree $G=(\Tree,\vertex).$ Key here is that our external fields are {\em local}, in that $B_v$ only depends on the degree of the vertex $v$, and not on any other properties of the pre-limit graph $G_n$. Similarly,
	\eqn{
	\langle\sigma_{\vertex_n} \rangle^f_{B_{\vertex_n}^{\sss(G_n)}(t)}\convp \expec\big[\langle\sigma_{\vertex} \rangle^f_{B_{\vertex}^{\sss(G)}(t)}\big],
	}
By \cite[Lemma 3.1]{DomGiaHof10}, $\expec\big[\langle\sigma_{\vertex} \rangle^+_{B_{\vertex}^{\sss(G)}(t)}\big]$ is close to $\expec\big[\langle\sigma_{\vertex} \rangle^f_{B_{\vertex}^{\sss(G)}(t)}\big]$ for $t$ large, whenever all external fields $B_v$ for $v\in V(\Tree)$ satisfy $B_v\geq B_{\min}>0$.  This shows that 
	\eqn{
	\frac{\partial}{\partial h}\varphi^{\rm \sss eIsing}_n(\beta,k,h)\convp \phi(\beta,k,h),
	}
for an appropriate $\phi(\beta,k,h)$. Theorem \ref{thm-pressure-per-particle-locally-tree-like} follows by combining two further crucial steps. First, convergence of the pressure per particle $\frac{1}{n} \log Z_{G_n}^{\rm \sss eIsing}(\beta,k,h)$ follows from integrating from $h_0$ to $\infty$, and considering what happens for $h$ being close to infinity (where all spins wish to be equal). For this part, we refer to the analysis in \cite[Section 12.2.3]{Hofs25}. Secondly, we then have to identify the limit by proving that $\phi(\beta,k,h)=\frac{\partial}{\partial h}\varphi^{\rm \sss eIsing}_n(\beta,k,h)$, and show that it can be described directly as a solution to the message-passing fixed-point equations (see also the discussion in Section \ref{sec-Bethe} on the Bethe partition function).
\medskip

We refrain from giving more details, and refer to \cite[Chapter 11]{Hofs25} instead. In particular, we can follow the proof of  \cite[Theorem 11.3]{Hofs25} almost verbatim. In turn, this theorem closely follows the proof by Dembo, Montanari and Sun \cite{DemMonSun13}.
\bigskip

\paragraph{\bf Acknowledgement.} The work of RvdH was supported in part by the Netherlands Organisation for Scientific Research (NWO) through the Gravitation {\small{\sf NETWORKS}} grant no.\ 024.002.003. The work of RvdH is further supported by the National Science Foundation under Grant No. DMS-1928930 while he was in residence at the Simons Laufer Mathematical Sciences Institute in Berkeley, California, during the spring semester 2025. RvdH thanks Amir Dembo and Lenka Zdeborova for enlightening discussions. The work of Van Hao Can is funded  by Vietnam National Foundation for Science and Technology Development (NAFOSTED) under grant number 101.03-2023.34.


\bibliographystyle{abbrv}

\begin{thebibliography}{10}

\bibitem{AldLyo07}
D.~Aldous and R.~Lyons.
\newblock Processes on unimodular random networks.
\newblock {\em Electron. J. Probab.}, {\bf 12}(54):1454--1508, 2007.

\bibitem{AldSte04}
D.~Aldous and J.~Steele.
\newblock The objective method: probabilistic combinatorial optimization and
  local weak convergence.
\newblock In {\em Probability on discrete structures}, volume~{\bf 110} of {\em
  Encyclopaedia Math. Sci.}, pages 1--72. Springer, 2004.

\bibitem{BasDemSly23}
A.~Basak, A.~Dembo, and A.~Sly.
\newblock Potts and random cluster measures on locally regular-tree-like
  graphs.
\newblock ar{X}iv:2312.16008 [math.PR], 2023.

\bibitem{BenBorCsi23}
F.~Bencs, M.~Borb\'{e}nyi, and P.~Csikv\'{a}ri.
\newblock Random cluster model on regular graphs.
\newblock {\em Comm. Math. Phys.}, {\bf 399}(1):203--248, 2023.

\bibitem{BenLyoSch15}
I.~Benjamini, R.~Lyons, and O.~Schramm.
\newblock Unimodular random trees.
\newblock {\em Ergodic Theory Dynam. Systems}, {\bf 35}(2):359--373, 2015.

\bibitem{BenSch01}
I.~Benjamini and O.~Schramm.
\newblock Recurrence of distributional limits of finite planar graphs.
\newblock {\em Electron. J. Probab.}, {\bf 6}(23):13 pp. (electronic), 2001.

\bibitem{BisBorChaKot00}
M.~Biskup, C.~Borgs, J.~T. Chayes, and R.~Koteck\'y.
\newblock Gibbs states of graphical representations of the {P}otts model with
  external fields.
\newblock {\em J. Math. Phys.}, {\bf 41}(3):1170--1210, 2000.

\bibitem{Can17}
V.~Can.
\newblock Critical behavior of the annealed {I}sing model on random regular
  graphs.
\newblock {\em J. Stat. Phys.}, {\bf 169}(3):480--503, 2017.

\bibitem{Can19}
V.~Can.
\newblock Annealed limit theorems for the {I}sing model on random regular
  graphs.
\newblock {\em Ann. Appl. Probab.}, {\bf 29}(3):1398--1445, 2019.

\bibitem{CanGiaGibHof22}
V.~Can, C.~Giardin\`a, C.~Giberti, and R.~van~der Hofstad.
\newblock Annealed inhomogeneities in random ferromagnets.
\newblock {\em Phys. Rev. E}, {\bf 105}(2):Paper No. 024128, 7, 2022.

\bibitem{CojGalGolRavSteVig23}
A.~Coja-Oghlan, A.~Galanis, L.~Goldberg, D.~Ravelomanana,
  J.B.and~\v{S}tefankovi\v{c}, and E.~Vigoda.
\newblock Metastability of the {P}otts ferromagnet on random regular graphs.
\newblock {\em Comm. Math. Phys.}, {\bf 401}(1):185--225, 2023.

\bibitem{DemMon10b}
A.~Dembo and A.~Montanari.
\newblock Gibbs measures and phase transitions on sparse random graphs.
\newblock {\em Braz. J. Probab. Statist.}, {\bf 24}(2):137--211, 2010.

\bibitem{DemMon10a}
A.~Dembo and A.~Montanari.
\newblock {Ising models on locally tree-like graphs}.
\newblock {\em Ann. Appl. Probab.}, {\bf 20}(2):565--592, 2010.

\bibitem{DemMonSlySun14}
A.~Dembo, A.~Montanari, A.~Sly, and N.~Sun.
\newblock The replica symmetric solution for {P}otts models on {$d$}-regular
  graphs.
\newblock {\em Comm. Math. Phys.}, {\bf 327}(2):551--575, 2014.

\bibitem{DemMonSun13}
A.~Dembo, A.~Montanari, and N.~Sun.
\newblock Factor models on locally tree-like graphs.
\newblock {\em Ann. Probab.}, {\bf 41}(6):4162--4213, 2013.

\bibitem{DomGiaHof10}
S.~Dommers, C.~Giardin{\`a}, and R.~{\noopsort{Hofstad}{van der Hofstad}}.
\newblock Ising models on power-law random graphs.
\newblock {\em J. Statist. Phys.}, {\bf 141}(4):638--660, 2010.

\bibitem{DomGiaHof12}
S.~Dommers, C.~Giardin{\`a}, and R.~{\noopsort{Hofstad}{van der Hofstad}}.
\newblock Ising critical exponents on random trees and graphs.
\newblock {\em Comm. Math. Phys.}, {\bf 328}(1):355--395, 2014.

\bibitem{Fort72a}
C.~M. Fortuin.
\newblock On the random-cluster model. {II}. {T}he percolation model.
\newblock {\em Physica}, {\bf 58}:393--418, 1972.

\bibitem{Fort72b}
C.~M. Fortuin.
\newblock On the random-cluster model. {III}. {T}he simple random-cluster
  model.
\newblock {\em Physica}, {\bf 59}:545--570, 1972.

\bibitem{ForKas72}
C.~M. Fortuin and P.~W. Kasteleyn.
\newblock On the random-cluster model. {I}. {I}ntroduction and relation to
  other models.
\newblock {\em Physica}, {\bf 57}:536--564, 1972.

\bibitem{GalSteVigYan16}
A.~Galanis, D.~\v{S}tefankovi\v{c}, E.~Vigoda, and L.~Yang.
\newblock Ferromagnetic {P}otts model: refined \#{BIS}-hardness and related
  results.
\newblock {\em SIAM J. Comput.}, 45(6):2004--2065, 2016.

\bibitem{GiaGibHofJanMai25}
C.~Giardin{\`a}, C.~Giberti, R.~van~der Hofstad, G.~Janssen, and N.~Maitra.
\newblock Annealed {P}otts models on rank-1 inhomogeneous random graphs.
\newblock ar{X}iv:2502.10553 [math.PR], 2025.

\bibitem{Grim06}
G.~{Grimmett}.
\newblock {\em {The random-cluster model.}}
\newblock Berlin: Springer, 2006.

\bibitem{HelJenPer23}
T.~Helmuth, M.~Jenssen, and W.~Perkins.
\newblock Finite-size scaling, phase coexistence, and algorithms for the random
  cluster model on random graphs.
\newblock {\em Ann. Inst. Henri Poincar\'e{} Probab. Stat.}, {\bf
  59}(2):817--848, 2023.

\bibitem{Hofs24}
R.~{\noopsort{Hofstad}{van der Hofstad}}.
\newblock {\em Random graphs and complex networks. {V}olume 2}.
\newblock Cambridge Series in Statistical and Probabilistic Mathematics.
  Cambridge University Press, 2024.

\bibitem{Hofs25}
R.~{\noopsort{Hofstad}{van der Hofstad}}.
\newblock {\em Stochastic processes on random graphs}.
\newblock 2025+.
\newblock In preparation,
  see\\\url{http://www.win.tue.nl/~rhofstad/SaintFlour\_SPoRG.pdf}.

\bibitem{MezMon09}
M.~M\'ezard and A.~Montanari.
\newblock {\em Information, physics, and computation}.
\newblock Oxford Graduate Texts. Oxford University Press, Oxford, 2009.

\bibitem{Ruoz12}
N.~Ruozzi.
\newblock The {B}ethe partition function of log-supermodular graphical models.
\newblock {\em Advances in Neural Information Processing Systems}, {\bf 25},
  2012.

\bibitem{SlySun14}
A.~Sly and N.~Sun.
\newblock Counting in two-spin models on {$d$}-regular graphs.
\newblock {\em Ann. Probab.}, {\bf 42}(6):2383--2416, 2014.

\bibitem{WilSudWai07}
A.~Willsky, E.~Sudderth, and M.~J. Wainwright.
\newblock Loop series and {B}ethe variational bounds in attractive graphical
  models.
\newblock {\em Advances in neural information processing systems}, {\bf 20},
  2007.

\bibitem{YedFreWei05}
J.~Yedidia, W.~Freeman, and Y.~Weiss.
\newblock Constructing free-energy approximations and generalized belief
  propagation algorithms.
\newblock {\em IEEE Trans. Inform. Theory}, {\bf 51}(7):2282--2312, 2005.

\end{thebibliography}
\providecommand{\noopsort}[1]{}\def\cprime{$'$}

\end{document}